\documentclass[10pt,reqno]{cpamart1}     

\authorheadline{V.~Kozlov, S.~Radosavljevic, V.~Tkachev, U.~Wennergren}
\titleheadline{The Permanency of the Age-Structured Population Model...}

\usepackage{stackrel}
\usepackage{amsthm,amssymb,bm}
\usepackage{enumerate,graphicx}
\usepackage[utf8]{inputenc}
\usepackage{mathrsfs,dsfont}
\usepackage[colorlinks=true,linkcolor=blue,citecolor=red]{hyperref}

\usepackage{enumitem}
\usepackage{tikz}
\usetikzlibrary{arrows,positioning}
\tikzset{
    >=stealth'
}
\usetikzlibrary{patterns}

\theoremstyle{plain}
\newtheorem{theorem}{Theorem}[section]
\newtheorem{lemma}[theorem]{Lemma}
\newtheorem{corollary}[theorem]{Corollary}
\newtheorem{proposition}[theorem]{Proposition}
\newtheorem{theorem*}{Theorem}

\newtheorem{thmy}{Theorem}

\theoremstyle{definition}
\newtheorem{definition}[theorem]{Definition}

\theoremstyle{remark}
\newtheorem{remark}[theorem]{Remark}

\def\R#1{\mathbb{R}^{#1}}
\def\B{\mathscr{B}}
\def\Lop{\mathcal{L}}
\def\Kop{\bar{\mathcal{K}}}
\DeclareMathOperator{\supp}{\mathrm{supp}}

\DeclareMathOperator{\diag}{\mathrm{diag}}
\DeclareMathOperator{\spectrum}{\mathrm{spec}}

\def\half#1#2{\begin{matrix}\frac{#1}{#2}\end{matrix}}
\def\Char{\mathscr{C}}

\def\Kopp{\widetilde{\mathcal{K}}}

\def\Rlo{\widetilde{\mathscr{R}}_0}
\def\Rol{\mathscr{R}_0}

\def\skob#1{\langle #1 \rangle_N}

\begin{document}                        

\title{Permanency of the age-structured population model on several temporally variable patches}

\author{Vladimir Kozlov}{Department of Mathematics, Link\"oping University}
\author{Sonja Radosavljevic}{Department of Mathematics, Link\"oping University}
\author{Vladimir Tkachev }{Department of Mathematics, Link\"oping University}
\author{Uno Wennergren}{Department of Physics, Chemistry, and Biology, Link\"oping University}






\begin{abstract}
We consider a system of nonlinear partial differential equations that describes an age-structured population inhabiting several temporally varying patches. We prove existence and uniqueness of solution and analyze its large-time behavior in cases when the environment is constant and when it changes periodically. A pivotal assumption is that individuals can disperse and that each patch can be reached from every other patch, directly or through several intermediary patches.
We introduce the net reproductive operator and characteristic equations for time-independent and periodical models and prove that permanency is defined by the net reproductive rate for the whole system. If the net reproductive rate is less or equal to one, extinction on all patches is imminent. Otherwise, permanency on all patches is guaranteed. The proof is based on a new approach to analysis of large-time stability.
\end{abstract}

\maketitle



\tableofcontents



\section{Introduction}

Population permanency in a patchy environment is the result of a complex interaction between spatial heterogeneity and temporal variability of the environment, dispersal, density-dependence and population structure. Each of these factors have relative importance for population growth and it differs for terrestrial and aquatic species, large and small populations, plants and animals, vertebrates and invertebrates etc.; see for instance  \cite{BG}, \cite{KW1}, \cite{Rg1979}.

One way to theoretically approach the problem of population dynamics is by formulating mathematical models that incorporate internal and external factors of population growth. The literature on the population models with various level of complexity is quite vast and  detailed review is beyond the scope of this paper. We mention only some of the well-established models that have been developed over the years. Among the unstructured models, the Malthus model of exponential growth and the Verhulst logistic model are especially important. For the age-structured models with density-dependency or time-dependency we refer to \cite{Grt},  \cite{Chipot1}, \cite{Chipot2}, \cite{DGH}, \cite{Pruss81}, \cite{Webb1985}, \cite{Ianelli2014}, \cite{we1}, \cite{we2}. The common point for the age-structured models is that {\it the net reproductive rate} and {\it the characteristic equation} are used to determine permanency of a population.

The spatial structure has been recognized as one of the most important factors of growth. In this case, each individual’s birth and death rate
are dependent upon the habitat/patch where they are in the landscape. For simplicity, let a population inhabit a discrete space which consists of several patches. A source is a high-quality patch that yields positive population growth, while a sink is a low-quality patch and it yields negative growth rate. In isolation, every subpopulation has its own dynamics. Linking the patches by dispersal lead us to the source-sink dynamics, where all local subpopulations contribute to the unique global dynamics. For populations that inhabit several patches, possibility to move from one patch to another can be crucial for survival. For example, dispersal from a source to a sink can save the local sink subpopulation from extinction through the rescue effect and recolonization \cite{PA}, \cite{PDias}, \cite{AH1}. The influence of spatial heterogeneity in unstructured populations was studied in \cite{Allen}, \cite{Arditi2015}, \cite{CChen1}, \cite{CChen2}, \cite{DeAngelis2016}, \cite{DeAngelis2014}. The trade-off between competition and dispersal is investigated in \cite{Amrasekare2001} and the relation between dispersal pattern and permanency was discussed in \cite{Hastings2006}, \cite{JanY}.

The continuous age-structured models with spatial structure can be divided into classes. In the first type of models individuals occupy position in a spatial environment and spatial movement is typically controlled by diffusion or taxis processes \cite{Webb2008}. In the second class fits the models with several species or populations occupying different regions (`patches') accompanying with migration between them. The usual practice here is to have only two classes (immature and adults) and dispersion between a few (two or three) temporally unchangeable patches, as in e.g. \cite{SWZ}, \cite{Tak1}, \cite{Tak2}, \cite{Terry}, \cite{WXZ}.

In this paper, we provide a rigorous mathematical derivation of the results considering the existence and uniqueness of a solution in a fairly general form in presence of migrations. Inspired by the single-patch models we come to the  fundamental questions:
\begin{itemize}
\item Is it possible to define an analogue of the characteristic equation and the net reproductive rate for the several-patches model?
\item If so, can they be used for the analysis of the large-time behavior of the solution and for establishing the condition for the population's permanency?
\end{itemize}

The main contribution of the paper is in the rigorous proof that the both questions have affirmative answers in the constant, periodic and the general time-dependent case. The method that we use for the time-dependent cases allows us to consider fluctuations that are not necessarily small in amplitude. Besides, we use general results to discuss the real world problems, such as the survival of migrating species and pest control.

To set up the model, we follow the argument of \cite{we3} and \cite{we2}, and assume that a population is age-structured, density-dependent and inhabits $N$ temporally variable and different patches. A local subpopulation on each patch experiences intraspecific competition, which results in additional density-dependent mortality. Let   $n_k(a,t)$ denote the age distribution in the population patch $k$ at
time $t$ with the  corresponding  birth rate $m_k(a,t)$ and  the initial distribution of population $f_k(a)$. Then the assumption that only the members of the age class are competing led to the following  McKendrick-von Foerster type {balance equations}  \cite{we3}:
\begin{align}\label{genpr}
\frac{\partial{\mathbf{n}(a,t)}}{\partial t}+\frac{\partial{\mathbf{n}(a,t)}}{\partial a} &=  -\mathbf{M}(\mathbf{n}(a,t),a,t)\mathbf{n}(a,t)+\mathbf{D}(a,t)\mathbf{n}(a,t)
\end{align}
in the domain
\begin{equation}\label{defB}
\B:=\{(a,t)\in \R{2}:0<a<B(t), \,\, t>0\}
\end{equation}
subject to the \textit{birth law}
\begin{align}\label{genbc}
\mathbf{n}(0,t)=\int_0^{\infty}\mathbf{m}(a,t)\mathbf{n}(a,t)\,da, \quad t> 0,
\end{align}
and the \textit{initial age distribution}
\begin{align}\label{genic}
\textbf{n}(a,0)=\textbf{f}(a), \quad a>0.
\end{align}
Here  $B(t)>0$ denotes the \textit{maximal length of life} of individuals in population at age $t\ge0$,
\begin{align*}
\mathbf{n}(a,t)&=(n_1(a,t), \ldots, n_N(a,t))^t,\\
\mathbf{f}(a,t)&=(f_1(a,t), \ldots, f_N(a,t))^t,\\
\mathbf{m}(a,t)&=\diag (m_1(a,t), \ldots, m_N(a,t)),\\
\mathbf{M}(n(a,t),a,t)&=\diag (M_1(n_1(a,t),a,t), \ldots, M_N(n_1(a,t),a,t)).
\end{align*}
where $M_k(v_k,a,t)$ is the mortality rate of the population patch $k$, and the the dispersion matrix  $\textbf{D}(a,t)=(D_{kj}(a,t))_{1\le k,j\le N}$ describes the migration rates between patches: the coefficients $D_{kj}(a,t)$ define a proportion of individuals of age $a$ at age $t$ on patch $j$ that migrates to patch $k$. Then
$$
\mathbf{P}(t)=\int_0^{B(t)}\mathbf{n}(a,t)\,da
$$
is the total population at time $t$.

The predecessor of the present model in the single patch case $N=1$ is the model proposed by von Foerster \cite{vonFoerster}; a detailed analysis was given by Gurtin and MacCamy \cite{Grt} and Chipot \cite{Chipot1}, \cite{Chipot2}. A comprehensive treatment of this approach is given by Iannelli \cite{Iannelli95}. Pr\"uss \cite{Pruss81}, \cite{Pruss83} was the first to study a mathematical model of an $N$-species population with
age-specific interactions in absence of migration. By using the theory of semilinear evolution equations he established the well-posedness and the existence of an equilibrium solution under certain constraints on the birth and death rates. He also derived some (local or asymptotic) stability results for for the equilibrium solutions.

When  $\mathbf{D}(a,t)\equiv0$, migration between patches is absent,  and the system \eqref{genpr} splits into $N$ independent balance equations. This model under an additional assumption that $\mathbf{M}(a,t)$ is the logistic regulatory function \eqref{logistic} has recently been studied in \cite{we2}. The case $\mathbf{D}(a,t)\not\equiv0$ is much more challenging. In modeling the source-sink dynamics, fundamentally important is the fact that individuals can disperse and move from one patch to another. Migration, which in the biological terms means a round-trip from a birthplace, is particularly significant. Then it is natural to expect that the global  and asymptotic behaviour of solutions to \eqref{genpr}--\eqref{genic} is determined by both the sign pattern and the weighted graph associated with $\textbf{D}(a,t)$.

\textbf{Outline}. A summary of the mathematical framework and our main results are presented in Section~\ref{sec:main}. In Section \ref{sec:prel} we discuss an auxiliary model and  derive some preliminary results on the corresponding lower and upper solutions.
In Section \ref{sec:general} we prove the existence and uniqueness of a solution to the balance equations  (\ref{genpr})--(\ref{genic}) by reducing the original problem to a certain nonlinear integral equation. In  Section~\ref{sec:const} we define the associated characteristic equation and the maximal solution, and establish one of the key results of the paper: the net reproductive rate dichotomy. The remaining part of the paper is dedicated to the study of the asymptotic behavior and stability of the solution. We consider three cases: a constant environment (i.e. the time-independent case) in Section~\ref{sec:const}, a periodic environment in Section \ref{sec:periodic} and an irregularly changing environment (i.e. the general time-dependent case) in Section \ref{sec:irreg}.

\paragraph{Notations.}
For easy reference we fix  some standard notation used throughout the paper. $\R{N}_+$  denotes the positive cone $\{x\in \mathbb{R}^N:x_i\ge0\}$. Given $x,y\in \mathbb{R}^N$ we  use the standard vector order relation: $x\le y$ if $x_i\le y_i$ for all $1\le i\le n$,
$x< y$  if $x\le y$ and $x\ne y$, and $x\ll y$ if $x_i< y_i$ for all $1\le i\le n$. Given $x\in \R{n}$,
$$
\renewcommand\arraystretch{1.5}
\| x\|_{p}=
\left\{
  \begin{array}{ll}
    (\sum_{k=1}^N |x_k|^p)^{1/p}, & \hbox{$1\le p<\infty$;} \\
    \max_{1\leq k\leq N}|x_k|, & \hbox{$p=\infty$.}
  \end{array}
\right.
$$
In particular, if $D=D_{jk}$ is an $N\times N$-matrix we define $\|D_{jk}\|_p$ for any $1\le p\le \infty$ in an obvious manner identifying $D$ with an element of $\R{N^2}$. Given $E\subset \mathbb{R}^N$ and a continuous function $h:E\to \mathbb{R}$, we define
$$
\|h\|_{C(E)}:=\sup_{x\in E}\|h_k(x)\|_\infty.
$$

\section{Main results}\label{sec:main}

\subsection{The structure conditions}
Before providing the main results, we give a brief summary of the structure conditions imposed on the balanced equations \eqref{genpr}--\eqref{genic}. We always assume that $\mathbf{m}(a,t)$ and $\mathbf{D}(a,t)$ are continuous\footnote{In fact, with some minor modifications, all the main results remains true under a weaker assumption that the structural coefficients are rather $L^{\infty}$-functions.} for $(a,t)\in \bar\B$ and  $\mathbf{M}(v,a,t)$ is a  continuous function of $(v,a,t)\in \R{}\times \bar\B$. Furthermore suppose the following structure conditions hold:
\begin{enumerate}[label=\textrm{(H\arabic*}),series=lafter]

\item\label{Hain1} there exists $0<b_1<b$ such that
$b_1\le B(t)\le b$ for all $t\ge0$ and
\begin{equation}\label{LipB}
\sup_{0<t_1<t_2<\infty}\frac{B(t_2)-B(t_1)}{t_2-t_1}<1
\end{equation}

\item\label{Hain2} for any fixed $(a,t)\in \B$, $M_k(v,a,t)$ is a nonnegative nondecreasing function of $v$ for $v\ge 0$, and there exist real numbers $\mu_\infty>0$, $\gamma>0$, and  a function $p(a)\ge \mu_\infty$  such that
\begin{equation}\label{hain2}
\begin{split}
M_k(v,a,t)-M_k(0,a,t)&\ge p(a) v^{\gamma}, \quad \forall (v,a,t)\in \R{}_+\times \B.
\end{split}
\end{equation}

\medskip

\item\label{Hain3} $\|\mathbf{D}\|_{C(\B)}<\infty$ and $\mathbf{D}(a,t)$ is a \textit{Metzler matrix}:
\begin{equation}\label{Metzler}
D_{kj}(a,t)\ge 0, \qquad k\ne j;
\end{equation}
\item
\label{Hain4}
$\|\mathbf{m}\|_{C(\B)}<\infty $ and there exist $0<a_m<A_m<b_1$ such that
$$\supp \mathbf{m}\subset[a_m,A_m]\times \R{+}.
$$

\item
\label{Hain5} the function $\mathbf{f}(a)$ is continuous and  $\supp \mathbf{f}\subset [0,B(0))$.
\end{enumerate}

Let us briefly explain the above conditions from the biological perspective. Concerning \ref{Hain1}, one usually uses a more restrictive condition that $B(t)$ is a constant. Nevertheless, \eqref{LipB} is  a more reasonable assumption: it means that the maximal length of life of individuals $B(t)$ in a population  may depend on $t$ but it grows not faster then the time. Mathematically, \eqref{LipB} asserts  that  the boundary curve $B(t)$ is transversal to the characteristics of \eqref{genpr}.

The monotonicity assumption in \ref{Hain2} ensures that increase in age-class density increases the death rate and has a negative effect on population growth. The classical example of the density independent mortality rate $M_k(v,a,t)=\mu_k(a,t)\ge\mu_\infty>0$ is compatible with $\gamma=0$ in \ref{Hain2}. Another  example is the logistic type model \cite{we2} with
\begin{equation}\label{logistic}
M_k(v,a,t)=\mu_k(a,t)\bigl(1+\frac{v}{L_k(a,t)}\bigr),
\end{equation}
where $L_k(a,t)\in L^\infty(\B)$ is the regulatory function (carrying capacity); this example fits \ref{Hain2} for $\gamma=1$.

Concerning the Metzler condition in \ref{Hain3}, note that the dispersion coefficient $D_{kj}(a,t)$ expresses the proportion of population $n_k(a,t)$ that from patch $j$ goes to patch $k$,  which naturally yields that $D_{kj}\ge0$. Furthermore, according the support condition in \ref{Hain4}, the improper integral in (\ref{genbc}) is well-defined  and actually  is taken over the finite interval $[a_m,A_m]$ which lies within the domain of definition of $\mathbf{n}(a,t)$ for any fixed $t>0$.  The condition \ref{Hain5} is  a natural assumption that the initial distribution of population is bounded by the life length.

\paragraph{The accessibility condition} For  further applications we shall also need an additional assumption on the structure of the dispersion matrix $\mathbf{D}$. In order to formulate it, let us recall some relevant concepts. Given a Metzler matrix $A\in \R{N\times N}$, one can associate a directed graph  $\Gamma(A)$ with nodes labeled by $\{1,2,\ldots,N\}$ where an arc leads from $i$ to $j$, $i\ne j$,  if and only if $A_{ij}> 0$. The patch
$j$ is said to be reachable from $i$, denoted $i\rightsquigarrow j$, if there exists a directed path from $i$ to $j$. A digraph is called
\textit{connected from vertex} $i$ if $i\rightsquigarrow j$ for all $j\ne i$  \cite[p.~132]{Balakrishnan}.

Then a patch $k$ is said to be \textit{accessible at age} $a\ge0$ if the associated digraph $\Gamma(\textbf{D}(a,t))$ is connected from $k$ for any $t>0$.

The accessibility condition relies on the sign pattern of the corresponding dispersion matrix and can be readily obtained by the standard tools of nonnegative matrix theory \cite[Section~3]{MincBook}.

Now, notice that by \ref{Hain4}  the following value is finite:
$$
\bar a_k=\inf_{t>0}\sup \{ a: m_k(a,t)>0 \}\le A_m<\infty.
$$
From the biological point of view, $\bar a_k$ is  the \textit{maximal fertility age} in population $k$. Our last condition reads as follows:

\begin{enumerate}[label=\textrm{(H\arabic*}),resume*=lafter]

\item \label{Hain6}
For any $1\le  k\le N$  there exists $ 0< \beta_k< \bar a_k$ such that the patch $k$ is accessible at age $\beta_k$.

\end{enumerate}

In other words, \ref{Hain6} asserts that for any patch $k$ there is a moment $\beta_k>0$  such that a (composite) migration from any other patch $j$ to $k$ is possible \textit{within the reproductive period}. Namely, some biological studies indicate that there are many different causes for dispersal, such as response to environmental conditions, prevention of inbreeding or competition for mates, see for instance \cite{BowlerB}. Thus, one can think of differences with respect to life-history traits, genetics and demography between dispersers and residents. When it comes to demography, more often than not, dispersing females are young individuals in their reproductive age, see, e.g., \cite{Gaines}, \cite{GreenwoodH}. Very  old individuals usually do not engage in breeding dispersal, which is the topic of our study.





\subsection{The Net Reproductive Rate Dichotomy}
Let us denote by $\rho(t)=\textbf{n}(0,t)$  the {newborn function}, i.e. a vector-function whose components denote the number of newborns on each patch. Then, the problem (\ref{genpr})--(\ref{genic}) can be reduced to the integral equation
\begin{align}\label{newb}
\rho(t)=\mathcal{K}\rho(t)+\mathcal{F}\mathbf{f}(t),
\end{align}
where $\mathcal{K}$ and $\mathcal{F}$ are positive nondecreasing operators with bounded ranges and $\mathcal{F}\mathbf{f}(t)=0$ for large $t>0$. Our strategy for proving permanency results is as follows: we first establish the permanency results for time-independent and time-periodic coefficients, and then show that in the general situation, a solution of (\ref{newb}) can be well-controlled by these cases.

If the environment is constant then the model parameters are time-independent functions. Then it is reasonable to assume that the maximal life-time is constant: $B(t)\equiv b$ \cite{Grt}, \cite{Chipot1}.
Our approach relies on a fine control of large-time behaviour  of an arbitrary solution to (\ref{newb}) by nontrivial solutions of the associated \textit{characteristic equation}
\begin{align}\label{char:const}
\rho=\Kop\rho,
\end{align}
where the operator $\Kop$ is  given by the right hand side in (\ref{genbc}) for a time-independent solution to (\ref{genpr})  with a constant boundary condition $\textbf{n}(0)=\rho$. Clearly, $\rho=0$ is a (trivial) solution of the characteristic equation.

Our  goal is to establish when a nontrivial positive solution $\rho\gg0$ exists. A crucial tool here is the so-called \textit{maximal solution} of the characteristic equation, i.e. a solution $\theta$ of \eqref{char:const} such that for an arbitrary solution $\rho$ there holds $\rho\le \theta$. In particular, $\theta=0$ implies that the characteristic equation has only trivial solutions. We establish the existence of the maximal solution in Section~\ref{sec:max}.

Another important ingredient is the \textit{net reproductive operator}
$$
\mathscr{R}_0\rho=\int_0^{\infty}\mathbf{m}(a)\mathbf{Y}(a;\rho)\,da,
$$
where $\mathbf{Y}(a;\rho)$ is the unique solution of the linearized initial problem
$$
\frac{d \mathbf{Y}(a;\rho)}{da}=(-\mathbf{M}(0,a)+\mathbf{D}(a))\mathbf{Y}(a;\rho),
\quad \mathbf{Y}(0;\rho)=\rho\in \R{N}_+.
$$
We show that  under conditions \ref{Hain1}-\ref{Hain6}, $\Rol:\R{N}_+\to \R{N}_+$ is a strongly positive operator. By Perron–Frobenius theorem, its spectral radius $\sigma(\Rol)$ is equal to the largest positive eigenvalue. We call this value {\it the net reproductive rate}.

To motivate the latter definition, observe that in the single-patch case, the net reproductive rate $R_0$ is given by
$$
R_0=\int_0^{\infty}m(a)e^{-\int_0^a\mu(v)dv}\,da.
$$
It it related to the solution of the Euler-Lotka characteristic equations in the linear age-structured population model; see \cite{Ianelli2014}.
According to \cite{we2}, $R_0$ is related to the solution $\rho^*$  of the characteristic equation in the nonlinear age-structured model.
Namely, if $R_0\le 1$, then $\rho^*=0$ and the population is going to extinction, while for $R_0>1$, we have $\rho^*>0$ and the population is permanent. The same is obviously valid if there are several patches without migration (i.e. $\textbf{D}\equiv 0$): every local subpopulation behaves accordingly to the value of $R_0$ on the respective patch.

The main contribution of this paper is the following dichotomy result on the long-term dynamics of populations.

\begin{thmy}[The Net Reproductive Rate Dichotomy]\label{thA}
If $\sigma(\mathscr{R}_{0})\leq 1$ then $\theta= 0$ and the characteristic equation \eqref{char:const} has no nontrivial solutions. If $\sigma(\mathscr{R}_{0})> 1$ then $\theta \gg 0$ and $\theta $ is the only nontrivial solution of the characteristic equation.

If $\chi(t)$ is an arbitrary solution of \eqref{newb} then
\begin{itemize}
\item if $\sigma({\mathscr{R}}_0)\le 1$ then $\chi(t)\rightarrow 0$ and $\mathbf{P}(t)\rightarrow 0$ as $t\rightarrow\infty$,

\item if  $\sigma({\mathscr{R}}_0)> 1$ then $\chi(t)\rightarrow \theta$ and $\mathbf{P}(t)\rightarrow \int_0^{\infty}\varphi(a;\theta)\,da$  as $t\rightarrow\infty$, where $\varphi(a;\theta)$ is the solution of the initial problem
    \begin{equation}\label{phis}
    \frac{d}{da}\varphi(a;\theta)= -\mathbf{M}(\varphi(a;\theta),a,t)\varphi(a;\theta)+ \mathbf{D}(a,t)\varphi(a;\theta), \quad \varphi(0;\theta)=\theta.
    \end{equation}
\end{itemize}
\end{thmy}

Thus, the net reproductive rate $\sigma({\mathscr{R}}_0)$ effectively determines large time behavior of population on $N$ patches in a constant environment. Here, as in the single-patch case, $\sigma({\mathscr{R}}_0)\le 1$ implies extinction of a population on all patches, while $\sigma({\mathscr{R}}_0)> 1$ grants the global permanency of a population. We see that the dichotomy result for a multi-patch population is completely consistent with the single-patch case when the net reproductive operator $\Rol$  coincides with the multiplication by $R_0$ (thus implying $R_0=\sigma({\mathscr{R}}_0)$).

It is also important to emphasize that the function $\varphi(a;\theta)$ in \eqref{phis} is exactly the unique equilibrium point of the problem \eqref{genpr}, \eqref{genbc} provided that $\theta$ satisfies the characteristic equation. In other words, Theorem~\ref{thA} implies the global stability result: any solution of the principal model converges at infinity to the unique equilibrium point given by the characteristic equation.

The proof of  Theorem~\ref{thA}, along with certain related results, occupies Section~\ref{sec:const} and make an essential use of the auxiliary monotonicity results collected in Section~\ref{sec:prel} and functional theoretic properties of the integral equation \eqref{newb} given in Section~\ref{sec:general}. Our approach relies on the following steps and can be described as follows. First we associate to an arbitrary solution $\chi$ of (\ref{newb}) certain lower and upper monotone sequences. The existence of an upper sequence relies on the boundedness of the image of $\mathcal{K}$. The construction of a lower sequence is  more tricky and involves certain fine properties of the maximal solution and some previous auxiliary monotonicity results accompanying by the accessibility condition \ref{Hain6}. The main problem here is to control a nonzero asymptotic behaviour of the lower approximants as $t\to \infty$. Next, we show that the large-time behaviour of $\chi$ can be well controlled by the limits at infinity constructed monotone approximants. Furthermore, we are able to identify the common limits as the maximal solution $\theta$. This finally establishes that the constructed sequences converge to the equilibrium point of the original problem. Notice that the monotonicity of the lower and upper approximations is crucial because the convergence established in the first steps is valid only on any bounded interval.

\subsection{Two-side estimates of $\sigma({\mathscr{R}}_0)$ and $\theta$}
A life-history trade-off between reproduction and migration has been noted for many species, including migratory birds and some insects (see for example \cite{MZ}, \cite{CASW}, \cite{PAG}). This trade-off is caused by energy constraints because both reproduction and migration are energetically costly for organisms. Keeping the assumption that the environment is constant and using the specific form of the balance system, we investigate the consequences of this trade-off.

The fact that individuals do not reproduce during migration is biologically justified and mathematically it is stated as:
\begin{align}\label{cond:D}
\sum_{k=1}^ND_{kj}(a)\le 0, \quad 1\le j\le N.
\end{align}
The relation between dispersion coefficients means that some migrants that are leaving patch $j$ will eventually die before reaching patch $k$, but they will not give birth during migration.
Then, we establish in Section \ref{sec:est} the following two-side estimates for the net reproductive rate.

\begin{thmy}
Under additional assumption that  \eqref{cond:D} holds we have
\begin{align*}
\max_{1\le k\le N}\int_0^{\infty}m_k(a)e^{-\int_0^a(\mu_k(v)+|D_{kk}(v)|)dv}\,da &\le \sigma(\Rol)\le  \int_0^{\infty}m(a)e^{-\int_0^a\mu(v)dv}\,da,
\end{align*}
where $m(a)$ is the maximal birth rate and $\mu(a)$ is the minimal death rate on all patches.
\end{thmy}

In addition, in Proposition~\ref{pro:upper} below we establish a priori estimates for the net reproductive rate and for the maximal solution $\theta$.

\subsection{Periodically and irregularly changed environment}
Natural habitats are usually positively autocorrelated, see for example \cite{Steele}. Therefore, the assumption that the vital rates, regulating function and dispersal coefficients are changing periodically with respect to time is a reasonable approximation. In the study of the large-time behavior of a solution to equation (\ref{newb}) in a periodically changing environment, the pivotal role belongs to the characteristic equation
\begin{align*}
\rho(t)=\Kopp\rho(t),
\end{align*}
where the operator $\Kopp$ is given by the right hand side of (\ref{genbc}) and $\mathbf{n}(a,t)$ solves (\ref{genpr}) with a periodic boundary condition $\mathbf{n}(0,t)=\rho(t)$, $\rho(t+T)=\rho(t)$. We establish in Section \ref{sec:periodic} that the operator $\Kopp$ is absolutely continuous which allows us to extend the methods of Section~\ref{sec:const} to the periodic case. In particular, the corresponding net reproductive operator $\Rlo$ defined on space of periodic continuous functions  is strictly positive and its spectral radius $\sigma(\Rlo)$ is equal to the largest eigenvalue. We are also able to establish the corresponding dichotomy result for a periodic environment.

If the environment is changing irregularly, the structure parameters the principal model (\ref{genpr})--(\ref{genic}) can be estimated from above and below by nonnegative periodic functions. Using these periodic functions as a structure parameters for new models, we formulate two associated periodic problems. One of them is the best-case scenario and its solution is an upper bound for the original problem. The other is the worst-case scenario and its solution is a lower bound. In other words, a solution for the general time-dependent problem can be bounded for large values of $t$ by above and below by the solution to the associated periodic problems, as stated in Theorem \ref{main_gen}.

\subsection{Source-sink dynamics}

Using the source-sink dynamics it is possible to explain permanency of a population on several patches provided that at least one patch is a source and that all patches are connected by dispersion. In Section \ref{sec:sink-source} we assume that the environment is constant and consists of several patches. Then it is possible to show that survival of population on both patches is possible provided that emigration from the source is sufficiently small.

Furthermore, in Section \ref{sink-only}, we show that permanency is possible even if all patches are sinks provided that dispersion is appropriately chosen. This is especially important for migratory birds, since both of their habitats can be seen as sinks (one because of the low reproduction due to insufficient resources, and the other because of the high mortality in the winter). This example can be related to the results in \cite{JanY}, where a simple model is used for analysis of connection between population permanency and allocation of offspring in a population that lives on several patches. One of the results is that permanency is possible even if all patches are sinks.


\section{An auxiliary model}\label{sec:prel}

\subsection{Upper and lower solutions}
Below we establish some auxiliary monotonicity results for lower and upper solutions to a general system of ordinary differential equations
\begin{equation}\label{system1}
\Lop w:=\frac{d}{dx}w(x)-\mathbf{F}(w(x),x)=0, \qquad x\in [0,b),
\end{equation}
where  $\mathbf{F}(w,x):\R{N}\times [0,b)\to \R{N}$ is a locally Lipschitz function in $w\in\R{N}$ for any $x\in [0,b)$ satisfying the \textit{Kamke-M\"uller condition}, i.e. that the Jacobian matrix $DF(w,x)$ is a Metzler matrix, i.e.
\begin{equation}\label{metzler1}
\frac{\partial F_i(w,x)}{\partial w_j}\ge 0 \quad i\ne j
\end{equation}
for almost all $w\in \R{N}$ and all $x\in [0,b)$. We assume additionally  that $\mathbf{F}$ satisfies
\begin{equation}\label{fequal0}
\mathbf{F}(0,x)= 0 \text{\,\,\, for any } x\in [0,b).
\end{equation}
In particular, this implies that $w(x)\equiv 0$ is a solution of (\ref{system1}).

We shall also exploit  a weaker version of the concept of irreducibility. More precisely, let $\textbf{F}(w,x)=(F_1(w,x),\ldots, F_N(w,x) )$ be continuously differentiable with respect to $w$ and let
$
D\textbf{F}(w,x):=(\frac{\partial F_k(w,x)}{\partial w_j})
$
denote the corresponding Jacobi matrix. Then an index $k\in \{1,2,\ldots, N\}$ is said to be $\textbf{F}$-\textit{accessible}  at  $x\in [0,b)$ if the associated digraph $\Gamma(D\textbf{F}(w,x))$ is connected from $k$ for any $w$.

In this paper, we are mostly interested in the  particular case when
\begin{align}\label{leftshort}
\textbf{F}(w,x)=-\textbf{M}({w},x){w}+ \textbf{D}(x){w}, \quad x\in [0,b), w\in \R{N}.
\end{align}
then  $D\textbf{F}(w,x)=\textbf{D}(x)$ and   a patch $k\in \{1,2,\ldots, N\}$ is accessible  at age $x$ if  $\Gamma(\textbf{D}(x))$ is connected from $k$. Note also  that if $\mathbf{F}$ is defined  by (\ref{leftshort}) then  (\ref{metzler1}) is equivalent to that $D\mathbf{F}(w,x)=\mathbf{D}(x)$ is a Metzler matrix. In this case the condition (\ref{fequal0}) is trivially satisfied.

\begin{definition}
A locally Lipschitz function $w(x)$ is called an \textit{upper} (resp. \textit{lower}) solution to (\ref{system1}) if $\frac{d}{dx}w(x)\ge \mathbf{F}(w(x),x)$ (resp. $\frac{d}{dx}w(x)\le \mathbf{F}(w(x),x)$) holds for all $x\in[0,b)$.
\end{definition}

The next lemmas generalize the corresponding facts for the cooperative system (cf. \cite[Remark~1.2]{Hal}) on lower (upper) solutions of (\ref{system1})  with Lipschitzian $\mathbf{F}$.  Notice also that our proofs are somewhat different from those given in \cite{Hal}. Let us agree to write
$$
v\ge_{k} u \quad \Leftrightarrow \quad
v\ge u \text{ and }v_k=u_k \text{ for some }1\le k\le N.
$$
First notice that $\mathbf{F}$ satisfies the so-called  quasimonotone condition \cite{HirschSmith}, \cite{Hal}.

\begin{lemma}\label{lem:metzler}
If $\mathbf{F}$ satisfies the Kamke-M\"uller  condition then $u\le_{k} v$ implies $F_k(u,x)\le F_k(v,x)$ for any $x\in [0,b)$.
\end{lemma}

\begin{proof}
Indeed, the function $g(t)=\mathbf{F}(u+t(v-u),x)$ is absolutely continuous in $[0,1]$, hence applying by the fundamental theorem of calculus and (\ref{metzler1}) that
\begin{equation}
\begin{split}
F_k(v,x)-F_k(u,x)&=\int_{0}^1g'_k(t)dt\label{fund}\\
&=\int_{0}^1 \sum_{i=1}^N\frac{\partial F_k(u+t(v-u),x)}{\partial w_i}(v_i-u_i) dt\\
&=\sum_{i=1,i\ne k}^N(v_i-u_i)\int_{0}^1\frac{\partial F_k(u+t(v-u),x)}{\partial w_i} dt\ge0,\\
\end{split}
\end{equation}
as desired.
\end{proof}

\begin{lemma}\label{lemma2}
Let $ w(x)$  be an upper solution of $(\ref{system1})$ a.e. in $[0,b)$ such that $w(0) \geq 0$. Then $w(x)\geq 0$ on $[0,b)$. Furthermore, if $w_j(0)>0$ then $w_j(x) > 0$ for  $x\in [0,b)$.
\end{lemma}

\begin{proof}
First we claim that $w(x)^-:=(w_1^-(x),...,w_N^-(x))$ is also an upper solution of (\ref{system1}) a.e. in $[0,b)$, where $w_k^-(x)=\min(0,w_k(x))$. Indeed, since each $w^-_k(x)$ is a locally Lipschitz function, there exists a full Lebesgue measure subset  $E\subset(0,b)$  where all $w^-_k(x)$ are differentiable. We will show that $w^-$ satisfy $(w^-)'(x)\ge \mathbf{F}(w^-(x),x)$ on $E$. Let $x_0\in E$ and $1\le k\le N$. If $w_k(x_0)\ge 0$ for some $k$ then $w_k^-(x_0)=0$, hence $x_0$ is a local maximum of $w_k^-(x)$ (because $w_k^-(x)\le 0$ everywhere). This yields $(w_k^-)'(x_0)=0$. Furthermore, since $0\ge_{k} w_-(x_0)$, we have by Lemma~\ref{lem:metzler} and (\ref{fequal0}) that
$$
(w_k^-)'(x_0)=0=F_k(0,x_0)\ge F_k(w^-(x_0),x_0).
$$
If $w_k(x_0)< 0$  then by the continuity of $w_k(x)$ one has $w_k^-(x)=w_k(x)$, $(w_k^-)'(x)=w_k'(x)$ in some neighbourhood of $x_0$. Thus, applying (\ref{system1}) we have by $w(x)\ge_{k}w^-(x)$ and Lemma~\ref{lem:metzler} that
\begin{align*}
(w_k^-)'(x)=w_k'(x)
\ge F_k(w(x),x)
\ge  F_k(w^-(x),x)
\end{align*}
holds everywhere in the neighbourhood of $x_0$. Thus, the claim is proved.

We also claim is that any upper solution to (\ref{system1}) with $w(0)= 0$ and $w(x)\le0$ for $x\in [0,b)$ is identically zero in the interval. Indeed, if $w$ is such a  function then let $T$ be chosen as the supremum of all $t\in[0,b)$ such that $w(x)=0$ in $[0,t]$. If $T=b$ the claim is proved. Therefore assume that $T<b$. Then by the continuity $w(T)=0$ and for any $\epsilon>0$ there exists $x\in [T,T+\epsilon]$ such that $w(x)<0$, and thus $\|w(x)\|_1>0$. Since $\mathbf{F}(w,x)$ is locally Lipschitz in $w$, there exist $M>0$ and $\epsilon>0$ such that $\|\mathbf{F}(w,x)-\mathbf{F}(0,x)\|_{1}\le M\|w\|_{1}$ for any $\|w\|_1<\epsilon$ and any $x\in [0,b)$. Define $h(x)=\|w(x)\|_1\equiv -\sum_{i=1}^N w_i(x)$ (recall that by the assumption $w_i(x)\le0$ for all $i$ and $x\in [0,b)$). By the continuity of $w(x)$, there exists $\delta$ such that $\|w(x)\|_1<\epsilon$ for any $|x-T|<\delta$. Let the set $E$  be defined as above and $x\in [T,T+\delta)$. Since by (\ref{fequal0}) $\mathbf{F}(0,x)=0$, we have
$$
h'(x)=-\sum_{i=1}^N w'_i(x)\le -\sum_{i=1}^N F_i(w(x),x)\le M\|w(x)\|_1=Mh(x).
$$
The latter inequality yields $(h(x)e^{-Mx})'\le0$ a.e. in $[T,T+\delta]$. Since $h(x)$ is locally Lipschitz it is absolutely continuous, thus $h(x)e^{-C(a)x}\le h(T)=0$  in $[T,T+\delta]$, i.e. $\|w(x)\|_1\equiv 0$ in the interval, a contradiction with the choice of $T$. This yields the  claim.

Now, if $w(x)$ is an upper solution to (\ref{system1}) with $w(x)\ge0$ then by the first claim $w^-(x)$ is an upper solution solution  with $w^{-}(0)=0$. Then the second claim implies $w^-(x)\equiv 0$ in $[0,b)$, thus we have $w(x)\ge0$ in $[0,b)$.

To finish the proof, let us suppose that $w_j(0)>0$. Since $F_j(y,x)$ is locally Lipschitz in $y$, for any $r>0$ there exists $C(r)$ such that (in virtue of (\ref{fequal0})) $|F_j(y,x)|\le C(r)\|y\|_1$ for all $y\in \R{N}$ and  $\|y\|\le r$. Let $0<\beta<b$ be chosen arbitrarily and let $r=\sup_{x\in [0,\beta]}|w_j(x)|$. Since $w(x)\ge_j w_j(x)e_j$, where $e_j$ is the $j$th coordinate vector, Lemma~\ref{lem:metzler} and the nonnegativity of $w_j(x)$ yield that
$$
\frac{d}{dx}w_j(x)\ge F_j(w(x),x)\ge F_j(w_j(x)e_j,x)\ge -C(r)w_j(x), \qquad x\in [0,\beta].
$$
The latter yields $w_j(x)e^{C(r)x} \ge w_j(0)>0$, thus $w_j(x)>0$ for every $x\in [0,\beta]$, and therefore in the whole interval $[0,b)$.
\end{proof}

\begin{lemma}\label{lemma20}
Let  $ w(x)$  be an upper solution of $(\ref{system1})$ with $w(0)> 0$ and such that the $k$-th patch is $\mathbf{F}$-accessible at some $\beta\in [0,b)$ then  $w_k(x) > 0$ on $(\beta,b)$.
\end{lemma}

\begin{proof}
It follows from Lemma~\ref{lemma2} that if $w_k(\beta)>0$ then  $w_k(x)>0$ holds everywhere in $[\beta,b)$. Therefore we may without loss of generality assume that $w_k(\beta)=0$. Let us suppose by contradiction that there exists $\beta_1\in (\beta,b)$ such that $w_k(\beta_1)=0$. Then  $w_k(x)\equiv 0$ in $[0,\beta_1]$. In particular, $w_k'(\beta)=0$. Since $w(0)>0$,  there exists $j$ such that $w_j(0)>0$  and, thus, $w_j(\beta)>0$. By the assumption, there exists a directed path $k\rightsquigarrow j$ in the graph $\Gamma(D\mathbf{F}(w,\beta))$. Equivalently, there exists  a sequence of pair-wise distinct $j_0=k$, $j_1,\ldots,j_{s-1}$, $j_m=j$ such that
\begin{equation}\label{conditi}
\frac{\partial F_{j_{i}}}{\partial w_{j_{i+1}}}(w(\beta),\beta)>0,\qquad \forall  i=0,1,\ldots, s-1.
\end{equation}
For any $i=0,\ldots, s-1$, let us define
$$
v_i=w(\beta)-(w_{j_0}(\beta)e_{j_0}+\ldots+w_{j_i}(\beta)e_{j_i}),
$$
where $e_i$ denotes the $i$th coordinate unit vector in $\R{N}$. Then
\begin{equation}\label{wbeta}
w(\beta)=v_0\ge_{j_0} v_1\ge_{j_1} \ldots \ge_{j_{s-1}} v_{j_s}=v_{j}\ge 0.
\end{equation}
Therefore by (\ref{system1}) and Lemma~\ref{lem:metzler} it follows for $j_0=k$ that
$$
0=w_{j_0}'(\beta)\ge F_{j_0}(v_0,\beta)\ge F_{j_0}(v_1,\beta)\ge F_{j_0}(0,\beta)=0,
$$
hence $F_{j_0}(v_0,\beta)= F_{j_0}(v_1,\beta)=0$. Arguing as in (\ref{fund}) we find
\begin{equation}\label{lastin}
\begin{split}
0&=F_{j_0}(v_0,\beta)-F_{j_0}(v_1,\beta)\\
&=\sum_{i=1,i\ne j_0}^N(v_0-v_1)_i\int_{0}^1\frac{\partial F_{j_0}(v_0+t(v_0-v_1),\beta)}{\partial w_i} dt\\
&\ge0.
\end{split}
\end{equation}
It follows from (\ref{lastin}), the nonnegativity of $(v_0-v_1)_{i}$ and the partial derivatives (for $i\ne j_0$) that all summands of  the latter sum must vanish. Since the integrands are non-negative continuous functions, they must vanish identically for $t\in [0,1]$. In particular, (\ref{conditi}) readily implies  that $(v_0-v_1)_{j_1}=0$. Thus, $w_{j_1}(\beta)=0$, and by the above we have $w'_{j_1}(\beta)=0$

Repeating  the same argument for the pair $(j_1,j_2)$ etc. implies $w_{j_2}(\beta)=0$ etc., thus yielding that $w_{j_s}(\beta)=w_{j}(\beta)=0$, a contradiction follows.
\end{proof}

%

\begin{proposition}[Comparison principle]\label{pro:comparison}
Let $u(x)$ and $v(x)$ be resp. upper and lower solutions to $(\ref{system1}))$ such that $u(0) \geq v(0)$. Then $u(x)\geq v(x)$  for all $x\in[0,b)$. If additionally the patch $k$ is $\mathbf{F}$-accessible at some $\beta\in [0,b)$    and $u(0) > v(0)$ then $u_k(x)> v_k(x)$  for all $x\in(\beta,b)$. If particular, if $(\ref{system1})$ is irreducible  and $u(0) > v(0)$  then $u(x)\gg v(x)$  for all $x\in(0,b)$.
\end{proposition}

\begin{proof}
Let $w(x)=u(x)-v(x)$. Then
$$
w'(x)\ge \mathbf{F}(v(x)+w(x),x)-\mathbf{F}(v(x),x)=\mathbf{G}(w(x),x),
$$
i.e. $w(x)$ is an upper solution to $\Lop_G w:=\frac{d}{dx}w(x)-G(w(x),x)$
with $G(\xi,x):=\mathbf{F}(v(x)+\xi,x)-\mathbf{F}(v(x),x)$. We have for the corresponding Jacobi matrices
$$
DG(\xi,x)=D \mathbf{F}(\xi+v(x),x),
$$
i.e. $\Lop$ and $\Lop_g$ satisfy simultaneously  the Kamke-M\"uller  condition. This readily yields the first claim of the proposition.

Now suppose that for some $k$ and $\beta\in [0,b)$ the associated digraph $\Gamma(D\mathbf{F}(w(\beta),\beta))$ is connected from $k$ and $u(0) > v(0)$. Since $D\mathbf{G}(w(\beta),\beta)=D\mathbf{F}(u(\beta),\beta)$ the digraph $\Gamma(D\mathbf{G}(w(\beta),\beta))$ is also connected from $k$. Applying Lemma~\ref{lemma20} we deduce $w_k(x)>0$, i.e.  $u_k(x)> v_k(x)$ for all  $x\in(\beta,b)$, as desired.
\end{proof}

\begin{corollary}\label{cor:comparison}
Let $u(x)$ be an lower (resp. upper) solution to $(\ref{system1})$. If $u(0)\le 0$ (resp $u(0)\ge0$) then  $u(x)\leq 0$ (resp. $u(x)\geq 0$ ) for all $x\in[0,b)$. If additionally the patch $k$ is $\mathbf{F}$-accessible at some $\beta\in [0,b)$    and $u(0) < 0$ (resp. $u(0)>0$) then $u_k(x)< 0$ (resp. $u_k(x)> 0$) for all $x\in(\beta,b)$.
\end{corollary}

\begin{proof}
Follows immediately from the fact that  $w(x)\equiv 0$ is a solution of (\ref{system1}).
\end{proof}

\begin{proposition}[Existence and Uniqueness]\label{pro:exists}
Let $(\ref{system1})$ satisfy the Kamke-M\"uller  condition and  there exists $C(\mathbf{F})>0$ such that
\begin{equation}\label{concavef}
\max_{k}F_k(w,x)\le C(\mathbf{F})\|w\|_\infty , \qquad \forall w\in \R{N}_+,\, x\in [0,b).
\end{equation}
Then for any $\xi\in \R{N}_+$  there exists a unique solution $ w(x)\in C^1([0,b),\R{N}_+)$ of $(\ref{system1})$ with $w(0)=\xi$. Furthermore, if $w(x)$ is a nonnegative lower solution to $(\ref{system1})$ then
\begin{align}\label{est1f}
\|w(x)\|_\infty\le \|w(0)\|_\infty e^{C(\mathbf{F})b}.
\end{align}
\end{proposition}

\begin{proof}
By the Cauchy-Peano Existence Theorem, (\ref{system1}) has a unique solution $w(x)$ in some  interval $[0,\beta )$, $0<\beta \leq b$. By Lemma~\ref{lemma2}, $w(x)\ge0$ for any $x\ge0$ in the domain of the definition. Let $[0,b')$ be the maximal interval of existence of the solution:
$$
b':=\sup\{\beta>0: \text{ there exists a solution of (\ref{system1}) on }[0,\beta)  \}.
$$
We claim that $b' =b$. It suffices to show that a solution $w(x)$ is uniformly bounded on any existence interval $[0,\beta)$, i.e. there exists $M>0$ such that for any $\beta<b'$ the inequality $\|w(x)\|_{\infty}\leq M$ holds in $[0,\beta)$. To this end, we make a more general assumption, that $w(x)$ is a \textit{nonnegative} lower solution to (\ref{system1}) on $[0,\beta)$ and consider
$$
H(x)=\|w(x)\|_{\infty}=\max_{k}w_k(x).
$$
In particular,  $H(x)$ is locally Lipschitz  on $[0,\beta)$, and thus a.e. differentiable there. Then for any point of differentiability  $x$ of $H$ there exists $k$ such that $H(x)=w_k(x)$ and $H'(x)=w'_k(x)$. We have $w(x)\le_k  H(x)\mathbf{1}$ which implies by Lemma~\ref{lem:metzler} and (\ref{concavef}) that
\begin{equation}\label{Hprim}
H'(x)=w'_k(x)\le F_k(w(x),x)\le F_k(H(x)\mathbf{1},x)\le C(\mathbf{F})H(x).
\end{equation}
Integrating the latter inequality (note that $H$ is absolutely continuous) yields
$$
H(x)\le H(0)e^{C(\mathbf{F})x}\le \|w(0)\|_\infty e^{NC(\mathbf{F})b}.
$$
This proves (\ref{est1f}). Furthermore, since the latter upper bound is independent of $\beta$, this implies $b'=b$, and thus the existence and the uniqueness of solution of (\ref{system1}) on $[0,b)$.
\end{proof}

\subsection{Further estimates for concave $\mathbf{F}$}

To proceed we consider some additional assumptions on $\mathbf{F}$. Namely, a vector-function $F\in C(\R{N},\R{N})$ is said to be \textit{concave} if
\begin{equation}\label{fcond0}
\mathbf{F}(\alpha_1 u+\alpha_2 v)\le \alpha_1 \mathbf{F}(u)+\alpha_2 \mathbf{F}(v), \qquad \forall \alpha_i\ge 1, \,\,u,v\in \R{N}.
\end{equation}
A concave vector-function $\mathbf{F}$ is said to be \textit{strongly concave} if for any $\alpha>1$  and any $u\ge 0$ with $u_k>0$ there holds
\begin{equation}\label{strongcon}
F_k(\alpha  u)< \alpha F_k(u).
\end{equation}

\begin{corollary}\label{cor:norma}
Let $\mathbf{F}$ be a concave vector-function satisfying the Kamke-M\"uller  condition. Let $v(x)$ be a lower and $u(x)$ be an upper solutions of $(\ref{system1})$. Then $v(x)-u(x)$ is a lower solution of $(\ref{system1})$.
\end{corollary}

\begin{proof}
The claim follows from (\ref{fcond0}) with $\alpha_1=\alpha_2=1$:
$$
v'(x)-u'(x)\le \mathbf{F}(v(x),x)-\mathbf{F}(u(x),x)\le \mathbf{F}(v(x)-u(x),x).
$$
\end{proof}

\begin{corollary}\label{compnorm}
Let $\mathbf{F}$ be a concave vector-function satisfying the Kamke-M\"uller  condition and $(\ref{concavef})$.
If  $v(x)$, $u(x)$ are solutions of $(\ref{system1})$ with $v(0)\ge0,u(0)\ge0$ then
\begin{equation}\label{normm}
\|  v(x)-  u(x)\|_{C[0,b)} \leq e^{C(\mathbf{F})b}\|v(0)-u(0)\|_{\infty}.
\end{equation}
\end{corollary}

\begin{proof}
By the assumptions $u(0),v(0)\in \R{N}_+$. First suppose that $v(0)\geq u(0)$ and define $w(x)=v(x)-u(x)$. Then by Proposition~\ref{pro:comparison}, $w(x)\geq 0$ for any $x\in[0,b)$. Therefore by Corollary~\ref{cor:norma} $w$ is a (nonnegative) lower solution to (\ref{system1}), thus by Proposition~\ref{pro:exists} we have
$\| w(x)\|_{C[0,b)} \leq e^{C(\mathbf{F})b}\|w(0)\|_{\infty},$ which proves (\ref{normm}).

In the general case, let $w(x)$ be the solution of (\ref{system1}) with the initial conditions $w_k(0)=\min(u_k(0),v_k(0))$, $1\le k\le N$. Then $u(0)\geq w(0)$ and $v(0)\geq w(0)$, hence by the above
\begin{equation}\label{norms}
\begin{split}
\|  u(x)-  w(x)\|_{C[0,b)} &\leq e^{C(\mathbf{F})b}\|  u(0)-  w(0)\|_{\infty}\\
\|  v(x)-  w(x)\|_{C[0,b)} &\leq e^{C(\mathbf{F})b}\|  v(0)-  w(0)\|_{\infty}.
\end{split}
\end{equation}
Since $u(x)\ge w(x)$ and $v(x)\ge w(x)$ for any $x\in [0,b)$ we also have
$$
w(x)-v(x)\le u(x)-  v(x)\le u(x)-w(x),
$$
which by virtue of (\ref{norms}) yields
$$
\|  u(x)-  v(x)\|_{C[0,b)}\le  e^{C(\mathbf{F})b}\max\{\|  u(0)-  w(0)\|_{\infty},\|  v(0)-  w(0)\|_{\infty}\}
$$
On the other hand, by our choice,  for any $k$ there holds that
$$
\max\{|u_k(0)-w_k(0)|,|v_k(0)-w_k(0)|\}=|u_k(0)-v_k(0)|,
$$
hence
$$
\max\{\|  u(0)-  w(0)\|_{\infty}, \|  v(0)-  w(0)\|_{\infty}\}\le \|u(0)-v(0)\|_\infty.
$$
which  yields (\ref{normm}).
%
\end{proof}

\begin{proposition}\label{pro:312}
Let $\phi(x,\xi)$ denote the solution $w(x)$ of problem $(\ref{system1})$ in $[0,b)$ with the initial condition $w(0)=\xi\in\R{N}_+$. Suppose $\mathbf{F}$ satisfy the Kamke-M\"uller condition and that it is concave. Then
\begin{equation}\label{conv2}
\phi(x,\alpha \xi)\le \alpha \phi(x,\xi), \qquad \forall \alpha\ge1, \forall x\in [0,b).
\end{equation}
Let additionally $\mathbf{F}(w,x)$ be strongly concave, $\xi> 0$ and $\alpha>1$. If the patch $k$ is $\mathbf{F}$-accessible at some $\beta\in [0,b)$ then
\begin{equation}\label{conv3}
\phi_k(x,\alpha \xi)< \alpha \phi_k(x,\xi), \qquad   \forall x\in (\beta,b).\quad\quad\quad
\end{equation}
\end{proposition}

\begin{proof}

Define $u(x)= \phi(x,\xi)$, $v(x)= \phi(x,\alpha\xi)$ and $w(x)=\alpha \phi(x,\xi)$. By the concavity condition,
\begin{equation}\label{wequat}
\Lop w=\Lop(\alpha u)\ge \alpha \Lop(u)\ge0
\quad \forall x\in [0,b),
\end{equation}
where $\Lop u=\frac{du}{dx}-\mathbf{F}(u(x),x)$.
In other words, $w(x)$ is an upper solution with
$$
w(0)=v(0)=\alpha\xi,
$$
hence Proposition~\ref{pro:comparison} yields $w(x)\ge v(x)$ for $x\in [0,b)$. This yields (\ref{conv2}).

Now, suppose that $\mathbf{F}(w,x)$ is strongly concave, $\xi> 0$, $\alpha>1$ and patch $k$ is $\mathbf{F}$-accessible at some $\beta\in [0,b)$. By virtue of (\ref{conv2}), it suffices to show that the equality $w_k(x)=v_k(x)$ is impossible in $(\beta,b)$. Arguing by contradiction let us assume that there exists $x_0\in (\beta,b)$ such that $w_k(x_0)=v_k(x_0)$. We claim that in this case $w_k(x)\equiv v_k(x)$ for any $x\in [\beta,x_0)$. Indeed, if not then there exists $x_1\in(\beta,x_0)$ such that $w_k(x_1)>v_k(x_1)$, hence the second part of Proposition~\ref{pro:comparison} implies $w_k(x)> v_k(x)$ for any $x\in (x_1,b)$, a contradiction at the point $x_0$ follows. Thus, $w_k(x)\equiv v_k(x)$ and, thus,
\begin{equation}\label{thus1}
\Lop w(x)=0 \text{ for any } x\in [\beta,x_0).
\end{equation}
On the other hand, by the assumption $u(0)=\xi> 0$ and Corollary~\ref{cor:comparison} we have $u_k(x)>0$ for $x\in (\beta,b)$. Using the  strong concavity condition (\ref{strongcon}), $F_k(\alpha  u(x),x)< \alpha F_k(u(x),x)$ for $x\in (\beta,b)$ which yields $\Lop w(x)=(\Lop(\alpha u))(x)> \alpha \Lop(u(x))=0$,  a contradiction with (\ref{thus1}) completes the proof.
\end{proof}

\section{The main representation}\label{sec:general}
We start with an auxiliary model (\ref{balanceD}) below and then prove the existence of a unique positive solution of (\ref{genpr})--(\ref{genic}) and examine asymptotic behavior of the obtained solution. Everywhere in this section we assume the conditions \ref{Hain1}--\ref{Hain4} are satisfied.

\subsection{The balanced equations}
Now we consider the particular  case of (\ref{system1}) with $\mathbf{F}(w,x)$ given by (\ref{leftshort}). In other words, we consider the differential operator
\begin{equation}\label{Loperator}
\Lop w(x)=\frac{d w(x)}{dx}+\mathbf{M}(w(x),x)w(x)- \mathbf{D}(x)w(x).
\end{equation}

For further applications, it is useful to specify the properties of $M_k$. Recall that in an important for us case of the Lotka-McKendrick-Von~ Foester model (\ref{genpr}) with (\ref{logistic}), i.e. each $M_k(v,x)$ is actually an \textit{increasing linear} function in $v$. Keeping on the monotonicity, we also impose some additional  growth conditions on $M_k$. Namely, we suppose that each $M_k(v,x)$ satisfies \ref{Hain2}, i.e. is a nonnegative  continuous function on $\R{}\times[0,b)$,
\begin{equation}\label{uslovM2}
M_k(v,x) \text{ is strongly increasing in $v\ge0$ for any fixed $x\in [0,b)$}
\end{equation}
and there exist $\gamma>0$ and $\mu_\infty>0$ such that
\begin{equation}\label{uslovM2gamma}
M_k(v,x)-\mu_k(x)\ge \mu_\infty v^{\gamma}, \quad \forall (v,x)\in \R{}_+\times [0,b),
\end{equation}
where
\begin{equation}
\mu_k(x):=M_k(0,x)\ge0.
\label{uslovM3}
\end{equation}

%

\begin{proposition}\label{pro:ex1}
Let $\Lop$ be given by $(\ref{Loperator})$ satisfying \ref{Hain2} and \ref{Hain3}. Then for any ${\xi}\in \R{N}_+$  there exists a unique solution $w(x)\in C([0,b),\R{N}_+)$ to  the initial value problem
\begin{equation}\label{balanceD}
\renewcommand\arraystretch{1.5}
\left\{
\begin{array}{rcl}
 \Lop w(x)&=&0  \qquad x\in[0,b)\\
  w(0) &=& \xi.
\end{array}
\right.
\end{equation}
The solution is nonnegative and bounded,
\begin{align}\label{est1}
0\leq w_k(x)\leq \|w(0)\|_\infty e^{N\|\mathbf{D}\|b},
\end{align}
and furthermore
\begin{align}\label{est2}
\|w(x)\|_{\infty} \le \|w(0)\|_{\infty} e^{N\|\mathbf{D}\|b-\int_0^x\mu (s)ds},
\end{align}
where $\mu(x)=\min_{k}\mu_k(x).$
\end{proposition}

\begin{proof}
Using the notation of (\ref{leftshort}), the Metzler property on $D$ implies that $\mathbf{F}$ satisfies the Kamke-M\"uller condition in $[0,b)$. Furthermore, since $M_k\ge 0$ one also has
$$
F_k(w,x)\le \|w\|_{\infty}\sum_{j=1}^N|D_{kj}(x)|\le N\|\mathbf{D}\|\|w\|_{\infty}
$$
which implies (\ref{concavef}) with $C(\mathbf{F})=N\|\mathbf{D}\|$. Thus, the assumptions of Proposition~\ref{pro:exists} are fulfilled. This yields the existence of the initial problem (\ref{balanceD}) and (\ref{est1}). Furthermore, if $H(x)=\|w(x)\|_{\infty}$ then by  (\ref{Hprim}) at any point $x\in [0,b)$ of differentiability of $H$
\begin{align*}
H'(x)&\le \max_{k}F_k(H(x)\mathbf{1},x)\\
&\le N\|\mathbf{D}\|H(x)-\min_{k}M_k(H(x),x)\\
&\le (N\|\mathbf{D}\|-\mu(x))H(x)
\end{align*}
which readily yields (\ref{est2}).
\end{proof}

\begin{proposition}[The Universal Majorant]\label{cor1}
Let $\Lop$ be given by $(\ref{Loperator})$ satisfying \ref{Hain2} and \ref{Hain3}.  Then any solution $\Lop w(x)=0$ satisfies
$$
w(x) \leq \omega_1 x^{-1/\gamma}\mathbf{1}_N\qquad x\in(0,b), \,\, 1\le k\le N,
$$
where
\begin{equation}\label{betadef}
\omega_1= \left(\frac{1+N\|\mathbf{D}\| b}{\gamma\mu_\infty}\right)^{1/\gamma}.
\end{equation}
\end{proposition}

\begin{proof}
Let us consider $h(x)=\omega_1 x^{-1/\gamma}\mathbf{1}_N$, where $\omega_1$ is defined by (\ref{betadef}). Then using  (\ref{uslovM2}) and $M_k(0,x)\ge0$ we have for any $k=1,\ldots,N$ and $x\in [0,b)$
$$
M_k(h_k(x),x)=M_k(\omega_1 x^{-1/\gamma},x)\ge \omega_1^{\gamma}\mu_\infty x^{-1},
$$
hence
\begin{align*}
(\Lop h(x))_{k}
&\ge -\frac{\omega_1}{\gamma}x^{-1-1/\gamma} +\mu_\infty\omega_1^{1+\gamma} x^{-1-1/\gamma}-N\|\mathbf{D}\|\omega_1 x^{-1/\gamma}\\
&\ge \frac{\omega_1}{\gamma}x^{-1-1/\gamma}(\gamma\mu_\infty\omega_1^{\gamma}-1-N\|\mathbf{D}\| x)\\
&\ge \frac{\omega_1}{\gamma}x^{-1-1/\gamma}(\gamma\mu_\infty\omega_1^{\gamma}-1-N\|\mathbf{D}\| b)\\
&\ge0,
\end{align*}
i.e. $h(x)$ is an upper solution. Now, if $w(x)$ be an arbitrary solution of $\Lop w=0$ then by (\ref{est1f}),  $w(x)$ is bounded on $[0,b)$:  $|w_k(x)|\le \|w(0)\|_\infty e^{C(\mathbf{F})b}$ for any  $k=1,\ldots,N$ and $x\in [0,b)$. Since $M_k\ge0$ one has $C(\mathbf{F})\le N\|\mathbf{D}\|$. Let $c=\|w(0)\|_\infty e^{N\|\mathbf{D}\|b}$ and $x_0:=\min\{(\omega_1/ c)^{\gamma},b\}$. Then $h(x)\ge w(x)$ on the whole interval $(0,x_0)$. This proves the claim if $x_0\ge b$. If $x_0<b$ then since $h(x)$ is an upper solution of (\ref{balanceD}) and $h(x_0)=c\ge w(x_0)$. Therefore Proposition~\ref{pro:comparison} yields $h(x)\ge w(x)$ for any $x\in (x_0,b)$, which finishes the proof.
\end{proof}



\subsection{The main represenation}\label{subMain}

\begin{lemma}
\label{lem:LipB}
Let $\B$ be defined by $(\ref{defB})$ and let
$$
\B^-=\{(a,t)\in \B:a>t\}, \qquad
\B^+=\{(a,t)\in \B:a<t\}.
$$
Then each of $\B^-$ and $\B^+$ is a connected open set.
\end{lemma}

\begin{proof}
It suffices to prove that for any $y\ge0$, the set $\{s\ge0: (s,y+s)\in \bar\B\}$ is connected. To this end let us suppose that $(0,y)\in \bar \B$ and let $S$ be the closed component of $\{s\ge0: (s,y+s)\in \bar\B\}$ containing $(0,y)$. Let $(s_1,y+s_1)$ be the right endpoint of $S$. Then $(s_1,y+s_1)\in \partial \B$. We claim that $(s,y+s)\in \R{2}\setminus \bar \B$ for $s>s_1$. Indeed, arguing by contradiction, one concludes that there exists $s_2>s_1$ such that $(s_2,y+s_2)\in \partial \B$. This yields $B(y+s_i)=s_i$, $i=1,2$, thus by (\ref{LipB})
$$
1=\frac{B(y+ks_2)-B(y+ks_1)}{s_2-s_1}<1.
$$
The contradiction yields our claim and, thus, the desired connectedness.
\end{proof}

%

Let us define
\begin{equation*}
\begin{split}
\B_1^+&=\{(x,y): x<B(x+y),\,y>0\}
\\
\B_1^-&=\{(x,y): 0<x<T,\,0<y<B(x)-x\},
\end{split}
\end{equation*}
as it is shown on Fig.~\ref{fig:Splus} and ~\ref{fig:Sminus}, respectively.

\begin{figure}[th]
\quad
\begin{tikzpicture}
  [scale=0.9,auto=left]
\fill[fill=gray!13] (3,0) -- plot[domain=0:4.1,smooth]({4.9-1*(\x-3)^2/10+0.2*sin((\x-3)*(\x/3+3.8) r)},\x) -- (3,4.1) -- cycle  ;
\draw[color=black,line width=1.3pt] plot[domain=2:4.12,smooth]({4.9-1*(\x-3)^2/10+0.2*sin((\x-3)*(\x/3+3.8) r)},\x);
\draw[color=black,line width=1.2pt] (3,0)--(3+2,0+2);
\node[] (X) at (7,-0.2) {$a$};
\node[] (Y) at (2.8,4.2) {$t$};
\node[]  at (4.2,3) {$\B^+$};
\node[] at (5.35,4) {$\partial\B$};
\node[]  at (5.7,2) {$(T,T)$};

\fill[fill=white!10] (3,0) -- plot[domain=0:2,smooth]({4.9-1*(\x-3)^2/10+0.2*sin((\x-3)*(\x/3+3.8) r)},\x) -- (5,2) -- cycle  ;

\draw[->] (3,-0.2) to (3,4.5);
\draw[->] (2.8,0) to (7,0);
\node[]  at (5,2) {$\bullet$};
\node[] (X) at (4.97,-0.3) {$T$};
\draw[dashed] (5,0) -- (5,2);
\end{tikzpicture}
\begin{tikzpicture}
\node[]  at (0,0) {};
\node[] (A) at (0,2) {};
\node[] (B) at (2,2) {};
\draw[->, line width=0.8pt] (A) to (B);
\node[] (B) at (1.3,3) {\(\left\{
                          \begin{array}{l}
                            a=x  \\
                            t=x+y
                          \end{array}
                        \right.
\)};

\end{tikzpicture}
\begin{tikzpicture}
  [scale=0.9,auto=left]
\fill[fill=gray!13] (3,0) -- plot[domain=0:4.12,smooth]({5-0.31*(\x^2/10+cos(3*\x+\x^2/2 r))*sin((\x-3)*(\x/3+3.8) r)},\x) -- (3,4.1) --  cycle ;
\draw[color=black,line width=1.3pt] plot[domain=0:4.12,smooth]({5-0.31*(\x^2/10+cos(3*\x+\x^2/2 r))*sin((\x-3)*(\x/3+3.8) r)},\x);

\node[] (X) at (7,-0.2) {$x$};
\node[] (X) at (4.8,-0.3) {$T$};
\node[] (Y) at (2.8,4.2) {$y$};
\node[] (Y) at (6.3,2.2) {$x=B(x+y)$};    
\node[] (Y) at (4,2) {$\B^+_1$};
\draw[->] (3,-0.2) to (3,4.5);
\draw[->] (2.8,0) to (7,0);
\end{tikzpicture}
\vspace*{-0.5cm}
\caption{The domains $\B^+$ and $\B^+_1$}
\label{fig:Splus}
\end{figure}

\begin{figure}[th]
\begin{tikzpicture}
  [scale=0.9,auto=left]
\fill[fill=gray!13] (3,0) -- plot[domain=0:2,smooth]({4.9-1*(\x-3)^2/10+0.2*sin((\x-3)*(\x/3+3.8) r)},\x) -- (3,4.1) -- cycle  ;
\draw[color=black,line width=1.3pt] plot[domain=0:2,smooth]({4.9-1*(\x-3)^2/10+0.2*sin((\x-3)*(\x/3+3.8) r)},\x);
\draw[color=black,line width=1.2pt] (3,0)--(3+2,0+2);
\node[] (X) at (7,-0.2) {$a$};
\node[] (Y) at (2.8,3.2) {$t$};
\node[] (B) at (4.0,-0.3) {$B(0)$};
\node[] (B) at (3.97,0.6) {$\B^-$};
\node[] (B) at (3.6,1.7) {$S$};
\fill[fill=white!20] (3,0)-- (5,2) -- (6,3) -- (5,4.1) -- (3,4.1) -- (3,2) -- cycle  ;
\draw[->] (3,-0.2) to (3,3.5);
\draw[->] (2.8,0) to (7,0);
\node[] (B) at (5.5,1) {$a=B(t)$};
\node[] (xi) at (5,2) {$\bullet$};
\node[] (xi1) at (5.1,2.4) {$(T,T)$};
\end{tikzpicture}
\begin{tikzpicture}
\node[]  at (0,0) {};
\node[] (A) at (0,2) {};
\node[] (B) at (2,2) {};
\draw[->, line width=0.8pt] (A) to (B);
\node[] (B) at (1.3,3) {\(\left\{
                          \begin{array}{l}
                            a=x+y  \\
                            t=x
                          \end{array}
                        \right.
\)};
\end{tikzpicture}
\begin{tikzpicture}[scale=0.9,auto=left]

\draw[color=black,line width=1.3pt] plot[domain=0:2,smooth](\x, {1*(2-\x/3-\x^2/2+\x^3/10)*cos(\x*3.1415/4 r)});
\node[] (X) at (4,-0.2) {$x$};
\node[] (Y) at (2.8-3,3.2) {$y$};
\fill[fill=gray!13] (0,0)-- (0,1.3) --
plot[domain=0:2,smooth](\x, {1*(2-\x/3-\x^2/2+\x^3/10)*cos(\x*3.1415/4 r)})
-- (2.12,0)  -- cycle  ;
\draw[->] (-0.2,0) to (4,0); 
\draw[->] (0,-0.2) to (0,3.3);
\node[] at (4.12,-0.3) {$\phantom{B(0)}$};
\node[] (T) at (2.12,-0.3) {$T$};
\node[] (B0) at (-0.6,1.9) {$B(0)$};
\node[] (T) at (0.6,0.4) {$\B^-_1$};
\node[] (T) at (2.52,1) {$y=B(x)-x$};
\end{tikzpicture}
\vspace*{-0.5cm}
\caption{The domains $\B^-$ and $\B^-_1$}
\label{fig:Sminus}
\end{figure}

Next, let $\Phi(x;\rho,y)$ denote respectively $\Psi(x;\mathbf{f},y)$  the solutions $h(x)$ of the  initial value problems
\begin{equation}\label{balanceV}
\renewcommand\arraystretch{1.5}
\left\{
\begin{array}{rl}
  \frac{d }{dx}{h}(x)&=-\mathbf{M}({h}(x),x,x+y){h}(x)+\mathbf{D}(x,x+y){h}(x), \\
  h(0)&= \rho(y),\quad \qquad (x,y)\in \B_1^+,
\end{array}
\right.
\end{equation}
respectively
\begin{equation}\label{balanceW}
\renewcommand\arraystretch{1.5}
\left\{
\begin{array}{rl}
  \frac{d}{dx}{h}(x)&=-\mathbf{M}({h}(x),x+y,x){h}(x) +\mathbf{D}(x+y,x){h}(x),\\
  h(0) &=\mathbf{ f}(y), \quad \qquad (x,y)\in \B_1^-.
\end{array}
\right.
\end{equation}

\begin{lemma}\label{propPhi}
Let $\rho\in C(\R{}_+,\R{N}_+)\cap L^\infty(\R{}_+,\R{N}_+)$ and let $\mathbf{f}\in C(\R{}_+,\R{N}_+)$  satisfy \ref{Hain5}. Then  $\Phi(x;\rho,y)$ (resp. $\Psi(x;\mathbf{f},y)$) is a nonnegative function non-decreasing in $\rho$ (resp. $f$). Furthermore,
\begin{align}
\label{Phibound}
\Phi(x;\rho,y)&\le e^{N\|\mathbf{D}\|b}\|\rho\|_\infty\\
\Phi(x;\rho,y) &\leq \omega_1 x^{-1/\gamma}\mathbf{1}_N , \quad x\ge0.\label{estphi}
\end{align}
where  $\omega_1$ is defined by $(\ref{betadef})$, and
\begin{equation}\label{LipPhi}
|\Phi_k(x; \rho,y)-\Phi_k(x; \rho^*,y)|\le e^{N\|\mathbf{D}\|b}\| \rho(y)- \rho^*(y)\|_{\infty},
\end{equation}
\begin{equation}\label{psizero}
 \Psi(x;\mathbf{ f},y)=0 \quad \forall x\ge0, y\ge B(0).
\end{equation}
\end{lemma}

\begin{proof}
It follows from Proposition~\ref{pro:ex1} that  (\ref{balanceV}) and (\ref{balanceW}) have a unique nonnegative solution. Next, given two arbitrary $\rho$ and $\rho^*$, let  $h(x)$ and $h^*(x)$ be the corresponding solutions of (\ref{balanceV}). If $\rho\geq \rho^*$ then  Proposition~ \ref{pro:comparison} imply $h(x)\ge h^*(x)$ for $x\ge0$ and the monotonicity $\Phi(x;\rho,y)\ge \Phi(x; \rho^*,y)$ follows. Similarly one shows the monotonicity of $\Psi$. Furthermore,  if $\rho(t)$ and $\rho^*(t)$ are  two arbitrary nonnegative vector-functions, then Corollary~\ref{compnorm} and Proposition~\ref{pro:ex1} yield
\begin{align*}
|\Phi_k(x; \rho,y)-\Phi_k(x; \rho^*,y)|&=|h(x)-h^*(x)| \le e^{N\|\mathbf{D}\|b}\|\rho(y)- \rho^*(y)\|_{\infty}.
\end{align*}
Proposition~\ref{cor1} implies  (\ref{estphi}). Finally,
by  \ref{Hain5} $f(x)=0$ for all $x>B(0)$. Then by the uniqueness of solution of (\ref{balanceW}), $\Psi(x;\mathbf{f},y)=0$ for all $y\ge B(0)$ and $x\ge0$.
\end{proof}


\begin{proposition}\label{prsol}
Let $\mathbf{n}(a,t)\in C^1(\overline{\B})$ be a solution to the problem \eqref{genpr}--\eqref{genic} and let $\rho(t)=\mathbf{n}(0,t)$. Then
\begin{align}\label{n1}
\renewcommand\arraystretch{1.5}
\mathbf{n}(a,t)=\left\{\begin{array}{ll}
         \Phi(a; \rho,t-a), & t>a, \\
         \Psi(a; \mathbf{f},a-t), & a\geq t,
        \end{array}\right.
\end{align}
and
\begin{align}\label{nb}
\rho(t) &=\int_0^t\mathbf{m}(a,t)\Phi(a; \rho,t-a)\,da + \int_t^{\infty}\mathbf{m}(a,t)\Psi(a;\mathbf{f},a-t)\,da.
\end{align}
\end{proposition}

\begin{proof}
First let $(a,t)\in \B$ and $t>a$. Then in the new variables $(a,t)=(x,x+y)$ one has $(x,y)\in \B^+_1$ and the initial value problem  (\ref{genpr})--(\ref{genic}) becomes (\ref{balanceV}) for $h(x)=\mathbf{n}(x,x+y)$. This yields $\mathbf{n}(x,x+y)=\Phi(x;\rho,y)$
for each $y>0$, thus, returning to the old variables yields
$\mathbf{n}(a,t)=\Phi(a; \rho,t-a)$ for any $ t>a>0$. This proves the first part of representation (\ref{n1}). The second part is similarly obtained by the change of variables $(a,t)=(x+y,x)$. Furthermore, the continuity of $\mathbf{n}(a,t)$ follows from  (\ref{n1}) and the standard facts on continuity of solutions on parameters.  Finally, plugging  (\ref{n1}) in (\ref{genbc}) yields  (\ref{nb}).
\end{proof}

\subsection{The  integral equation}
It is straightforward to see that  if $\mathbf{M}$, $\mathbf{D}$, $\mathbf{m}$ and $\mathbf{f}$ are sufficiently smooth functions, then the function $\mathbf{n}(a,t)$ in (\ref{n1}) is a classical solution of the boundary value problem $(\ref{genpr})-(\ref{genic})$ in $\B$. On the other hand, in application it is natural to assume that these functions are  merely continuous (or even measurable). In that case, one can interpret the representation (\ref{n1}) with $\rho$ satisfying (\ref{nb}) as a \textit{weak solution} of $(\ref{genpr})-(\ref{genic})$. Furthermore, since a solution $\rho(t)$ of the integral equation (\ref{nb}) completely determines the population dynamics $\mathbf{n}(a,t)$, it is natural to characterize all nonnegative solutions of (\ref{nb}) (with a given function $f$). To this end, we  observe that  (\ref{nb}) can be thought of as an (nonlinear) operator equation on $\rho$:
\begin{align}\label{n}
 \rho=\mathscr{L}_\mathbf{f}\rho:=\mathcal{K} \rho + \mathcal{F} \mathbf{f},
\end{align}
where  the operators $\mathcal{K}$ and $\mathcal{F}$ are defined  resp. by
\begin{align}\label{k}
\mathcal{K} \rho(t)&=\int_0^t\mathbf{m}(a,t)\Phi(a; \rho,t-a)\,da\\
\label{f}
\mathcal{F} \mathbf{f}(t)&=\int_t^{\infty}\mathbf{m}(a,t)\Psi(a;\mathbf{f},a-t)\,da.
\end{align}
In this section we treat some general properties of $\mathscr{L}_\mathbf{f}$.

We fix some notation which will be used throughout the remained part of the paper. Let $\omega_1$ be defined by (\ref{betadef}) and let
\begin{equation}\label{omegadef}
\omega_2 = \omega_1\|\mathbf{m}\|_\infty\int_{a_m}^{A_m}\frac{da}{a^{1/\gamma}},
\end{equation}
where $A_m, a_m$ are the constants from ~\ref{Hain4} and $\mathbf{m}=\mathbf{m}(a,t)=(m_1(a,t),\ldots,m_N(a,t))$ is the birth rate. Let us also consider the following subsets of $\R{N}_+$:
\begin{equation}
\label{omme}
\begin{split}
 Q^-&:=\{x\in \R{N}:\,\,\, 0\le x\le \omega_2\mathbf{1}_N\},\\
 Q^+&:=\{x\in \R{N}:\,\,\, x\ge \omega_2\mathbf{1}_N\}.
\end{split}
\end{equation}

\begin{lemma}\label{lemOmega}
Let $\ref{Hain4}$ be satisfied. Then the  operators $\mathcal{F}$ and $\mathcal{K}$ are positive on the cone of nonnegative continuous vector-functions $C(\R{}_+,\R{N}_+)$ and have bounded ranges:
\begin{align}\label{Omega}
\mathcal{K}&:C(\R{}_+,\R{N}_+)\to C(\R{}_+, Q^-), \\
\label{FM}
\mathcal{F}&:C(\R{}_+,\R{N}_+)\to \{h \in C(\R{}_+, Q^-): \supp h\subset [0,A_m]\times \R{}\},
\end{align}
Furthermore,
  $\mathcal{K}$ is non-decreasing and Lipschitz continuous on $C(\R{}_+,\R{N}_+)$.
\end{lemma}

\begin{proof}
It readily follows from the nonnegativity of $m$ and Lemma~\ref{propPhi} that $\mathcal{K}$ and $\mathcal{F}$ preserve the cone of nonnegative functions $C(\R{}_+,\R{N}_+)$ and non-decreasing there. Furthermore,  using \ref{Hain4} we have from (\ref{estphi})
$$
(\mathcal{K} \rho)_k(t) \leq  \int_{a_m}^{A_m}\frac{\omega_1 m_k(a,t)}{a^{1/\gamma}}\,da \leq \omega_1\|\mathbf{m}\|_\infty\int_{a_m}^{A_m}\frac{1}{a^{1/\gamma}}\,da =\omega_2 .
$$
 This yields (\ref{Omega}) and thus the boundedness of the range of $\mathcal{K}$. The corresponding property for  $\mathcal{F}$ is  established similarly. Next, by \ref{Hain4} $\mathbf{m}(a,t)\equiv 0$ for $a\ge A_m$, hence for any $t\ge A_m$
$$
\mathcal{F} \mathbf{f}(t)=\int_t^{\infty}\mathbf{m}(a,t)\Psi(a;\mathbf{f},a-t)\,da=0
$$
which implies (\ref{FM}). Finally, if $ \rho$ and $ \rho^*$ are bounded functions then by (\ref{LipPhi}),
\begin{align*}
|(\mathcal{K} \rho-\mathcal{K} \rho^*)_k(t)| &\leq \|\mathbf{m}\|_{\infty}\int_{a_m}^{A_m}|\Phi_k(a; \rho, t-a)-\Phi_k(a; \rho^*, t-a)|\,da \\
& \leq (A_m-a_m)\|\mathbf{m}\|_{\infty}e^{N\|\mathbf{D}\|b} \|  \rho- \rho^*\|_{\infty},
\end{align*}
which yields  that $\mathcal{K}$ is a Lipschitz continuous operator.
\end{proof}

\begin{proposition}\label{gen_sol}
Given an arbitrary $\mathbf{f}\in C(\R{}_+,\R{N}_+)$, there exists a unique  solution $\rho\in C(\R{}_+,\R{N}_+)\cap L^\infty(\R{}_+,\R{N}_+)$ of  $(\ref{n})$.
\end{proposition}

\begin{proof}
Let us consider the sequence $\{\rho^{(i)}\}_{0\le i\le \infty}$ defined recursively by
\begin{align}\label{it11}
 \rho^{(i+1)}= \mathcal{K}\rho^{(i)}+\mathcal{F} \mathbf{f}, \quad  \rho^{(0)}=0.
\end{align}
Since $\mathcal{F} \mathbf{f}\ge 0$, we have
\begin{align*}
 \rho^{(0)} &=0 \leq \mathcal{F} \mathbf{f}=  \rho^{(1)}, \\
 \rho^{(1)} &= \mathcal{F} \mathbf{f} \leq \mathcal{K} \rho^{(1)} +\mathcal{F} \mathbf{f}=  \rho^{(2)}. \end{align*}
 This shows that $\rho^{(i+1)}-\rho^{(i)}\ge0$ for $i=0,1$. Then combining
 \begin{align*}
 \rho^{(i+1)}-\rho^{(i)} &= \mathcal{K} \rho^{(i)} -\mathcal{K}\rho^{(i-1)} , \quad i\ge0,
\end{align*}
with the monotonicity of $\mathcal{K}$ implies by induction that $\rho^{(i+1)}-\rho^{(i)} \ge0$ for any $i\ge0.$ In other words, $\{\rho^{(i)}\}_{0\le i\le \infty}$ is a pointwise non-decreasing sequence. On the other hand, by (\ref{Omega}) and (\ref{FM}) this sequence is uniformly bounded:
$$
\rho^{(i+1)}= \mathcal{K}\rho^{(i)}+\mathcal{F} \mathbf{f}\le 2\omega_2\cdot\textbf{1}_N.
$$
This implies the the existence of the limit
\begin{equation}\label{Romega}
\rho:=\lim_{i\to \infty }\rho^{(i)}\leq 2\omega_2\cdot\textbf{1}_N.
\end{equation}
Using (\ref{LipPhi}) and (\ref{Phibound}) we obtain  for any $i\ge1$
\begin{align*}
|\rho^{(i+1)}_k(t) &- \rho^{(i)}_k(t)|\leq \|\mathbf{m}\|_{\infty} \int\limits_0^t |\Phi_k(a; \rho^{(i)},t-a)-\Phi_k(a; \rho^{(i-1)},t-a)|\,da\\
&\leq C \int\limits_0^t |\rho^{(i)}(t-a)-\rho^{(i-1)}(t-a)|\,da= C \int\limits_0^t |\rho^{(i)}(a)-\rho^{(i-1)}(a)|\,da,
\end{align*}
where $C=e^{Nb\|\mathbf{D}\|} \|\mathbf{m}\|_{\infty}$. On iterating the latter inequality we obtain using $\rho^{(1)}\le \rho$ and (\ref{Romega})
\begin{align*}
|\rho^{(i+1)}_k(t)- \rho^{(i)}_k(t)|
&\leq C^i \int\limits_0^t\int\limits_0^{a_1}...\int\limits_0^{a_{i-1}} \rho^{(1)}(a)da\,da_1...\,da_{i-1}
\le 2\omega_2\frac{C^it^i}{i!}
\end{align*}
therefore
\begin{align}\label{iteration}
|\rho^{(i+j)}_k(t)-\rho^{(i)}_k(t)| &\leq 2\omega_2\sum_{s=0}^{j-1}\frac{(Ct)^{i+s}}{(i+s)!} \leq 2\omega_2 e^{Ct}\frac{C^it^i}{i!}.
\end{align}
Therefore for any fixed  $T>0$ and $0<t<T$, the latter expression converges to 0 as $i\to \infty$ uniformly in $j\ge1$. This establishes that  $\rho^{(i)}\rightarrow \rho$ in $L^{\infty}((0,T), \R{N}_+)$ for each $T>0$. In particular, by (\ref{it11}) this implies that $\rho$ satisfies (\ref{n}).

In order to establish the uniqueness we assume  that $\rho$ and $\tilde{\rho}$ are two solutions to  (\ref{n}). The tautological iterations $\rho^{(i)}:=\rho$ and $\tilde{\rho}^{(i)}:=\tilde{\rho}$, $i=0,1,2,\ldots$ obviously satisfy (\ref{it11}) which by virtue of  (\ref{iteration}) yields
\begin{align*}
|\rho_k(t)-\tilde{\rho}_k(t)|  \leq 2\omega_2\frac{C^it^i}{i!}\rightarrow 0 \quad\mbox{as}\quad i\rightarrow\infty,
\end{align*}
thus $\rho(t)\equiv \tilde{\rho}(t)$. Finally, by  Lemma \ref{propPhi} $\Phi_k$ and $\Psi_k$ are continuous, which yields the continuity of operators $\mathcal{K}$ and $\mathcal{F}$, and, thus, all iterations given by {(\ref{it11})} are continuous and so is the limit $\rho$. This  completes the proof.
\end{proof}

\subsection{The convolution property of $\mathcal{K}$}

\begin{lemma} \label{lem:conv}
Let  $\rho\in C(\R{}_+,\R{N}_+)$ and $\rho(t)>0$ for $t\in [s_1,s_2]\subset \R{}_+$. Let for some $k$ there exists $\beta_k<\sup \supp m_k$ such that the patch $k$ is accessible at $\beta_k$. Then there exist $a_k,b_k$ such that $[a_k,b_k]\Subset \supp m_k$, $\beta_k\le a_k$,  and
$(\mathcal{K}\rho(t))_k> 0$ for all $t\in [s_1+a_k,s_2+b_k]$.
\end{lemma}

\begin{proof}
There are $\delta>0$ and points $a_k',b_k'$, $a_k'<a_k<b_k<b_k'$, such that (i) $m_k(a)\ge\delta$ for $a\in [a_k',b_k']$, and (ii)  the patch $k$ is accessible at  $\beta_k\le a_k$. By Lemma~\ref{lemma20} we have
$$
\delta_1:=\min_{\substack{s_1 \le y\le s_2 \\a_k'  \le a\le b_k'}}
                \Phi_k(a;\rho,y)>0,
$$
hence if $t\in [s_1+a_k,s_2+b_k]$ then
\begin{align*}
(\mathcal{K}\rho)_k(t)&=
\int_0^tm_k(a)\Phi_k(a; \rho,t-a)\,da \ge \delta\int_{a_k'}^{\min\{t,b_k'\}}\Phi_k(a; \rho,t-a)\,da\\
&\ge \delta\delta' \int_{\max\{a_k',t-s_2\}}^{\min\{t,b_k',t-s_1\}}\,da
=\delta\delta' \int_{\max\{a_k',t-s_2\}}^{\min\{b_k',t-s_1\}}\,da
\end{align*}
We claim that $(\mathcal{K}\rho(t))_k> 0$ for all $t\in [s_1+a_k,s_2+b_k]$.
Indeed, the function $\xi(t)=\min\{b_k',t-s_1\}-\max\{a_k',t-s_2\}$ is obviously concave and
\begin{align*}
\xi(s_1+a_k')&=\min\{b_k',a_k'\}-\max\{a_k',a_k'+s_1-s_2\}=0,\\
\xi(s_2+b_k')&=\min\{b_k',b_k'+s_2-s_1\}-\max\{a_k',b_k'\}=0,
\end{align*}
hence by the maximum principle $\xi(t)>0$ for any $t\in (s_1+a_k',s_2+b_k')$. This yields the desired conclusion.
\end{proof}

\section{Constant environment}\label{sec:const}
The model (\ref{genpr})--(\ref{genic}) is more complicated for analysis under the assumption that a population lives in a temporally variable environment because the structure parameters are functions of age and time.  In this section we analyze a constant environment, then in section~\ref{sec:periodic} we continue with a periodically changing environment, and finally in section~\ref{sec:irreg} we describe an irregularly changing environment. Throughout this section, we assume the conditions  \ref{Hain1}--\ref{Hain5} are fulfilled. Also, it is reasonable to assume that the maximal life-time is constant: $B(t)\equiv b$. This condition is natural and is commonly used for both finite and infinite values of $b$, see \cite{Grt}, \cite{Chipot1}, \cite{Cushing84}.

\subsection{The characteristic equation}
Under assumptions that the vital rates, carrying capacity and dispersion coefficients are \textit{time-independent} functions,
%
%
the system (\ref{genpr})--(\ref{genic}) becomes
\begin{equation}\left\{\begin{aligned}\label{sys1}
&\frac{\partial \textbf{n}(a,t)}{\partial t}+\frac{\partial \textbf{n}(a,t)}{\partial a} = -\textbf{M}(\textbf{n}(a,t),a) \textbf{n}(a,t) +\textbf{D}(a)\textbf{n}(a,t), \\
&\textbf{n}(0,t)=\int_0^{\infty}\textbf{m}(a)\textbf{n}(a,t)\,da, \\
&\textbf{n}(a,0) =\textbf{f}(a).
\end{aligned}
\right.
\end{equation}
According to Proposition \ref{prsol}, there exists a unique solution $\mathbf{n}(a,t)$ of the problem (\ref{sys1}) given by
\begin{align}\label{nPhiPsi}
\renewcommand\arraystretch{1.5}
\mathbf{n}(a,t)=\left\{\begin{array}{ll}
         \Phi(a; \rho,t-a), & a<t, \\
         \Psi(a;\mathbf{f},a-t), & a\geq t,
        \end{array}\right.
\end{align}
where the \textit{newborns function}
$$
\rho(t)\equiv (\rho_1(t),\ldots,\rho_N(t))^t=\mathbf{n}(0,t)=\int_0^{\infty}\mathbf{m}(a)\mathbf{n}(a,t)\,da,
$$
satisfies the following identity:
\begin{align}\label{rhoind}
{\rho}(t)=\int_0^t\mathbf{m}(a){\Phi}(a;\rho,t-a)\,da +\int_t^{\infty}\mathbf{m}(a){\Psi}(a;\mathbf{f},a-t)\,da.
\end{align}
Using the notation of (\ref{k}) and (\ref{f}), we have

\begin{proposition}\label{pro:stationary}
Let $\mathbf{n}(a,t)$ be the solution of the problem \eqref{sys1}. Then the newborns function $\rho(t)$ satisfies the integral equation
\begin{equation}\label{rhoind1}
\rho=\mathscr{L}_\mathbf{f}\rho:= \mathcal{K}\rho+\mathcal{F}\mathbf{f}.
\end{equation}
\end{proposition}

It is natural to study stationary (i.e. time independent) solutions of (\ref{rhoind1}). Indeed, since $m(a)$ has a compact support, it follows from (\ref{rhoind}) that $\mathcal{F}\mathbf{f}$ vanishes for large enough $t$. This yields that any solution of (\ref{rhoind1}) satisfies
\begin{equation}\label{rhoind2}
\rho(t)=(\mathcal{K}\rho)(t)\quad \text{  for all } t\ge A_m.
\end{equation}
In particular, it is easy to see that if $\rho$  has a limit $\rho_\infty=\lim_{t\to \infty}\rho(t)$ then $\rho_\infty$  itself is a stationary solution of (\ref{rhoind2}). In the next section we study the stationary solutions in more detail.

To make these observations precise, we introduce the following operator:
\begin{align}\label{char111}
\Kop \rho:=\int_0^{\infty}\mathbf{m}(a)\varphi(a;\rho)\,da
\equiv  \int_{a_m}^{A_m}\mathbf{m}(a)\varphi(a;\rho)\,da , \qquad \rho\in \R{N}_+,
\end{align}
where $\varphi(a;\rho)=(\varphi_1(a;\rho), \ldots, \varphi_N(a;\rho))^t$ is the unique solution of the  initial problem
\begin{equation}
\left\{\begin{aligned}\label{varphieq}
\frac{d \varphi(a;\rho)}{da}&=-\mathbf{M}(\varphi(a;\rho),a)\varphi(a;\rho) +\mathbf{D}(a)\varphi(a;\rho),\\
\varphi(0;\rho)&=\rho.
\end{aligned}
\right.
\end{equation}
In particular, this yields in the notation of (\ref{balanceV}) for any $\rho\in \R{N}_+$ that
 \begin{equation}\label{phiphi}
\varphi(a;\rho) \equiv\Phi(a;\rho,y)\quad  \text{ for any }y\in \R{}.
 \end{equation}

\begin{corollary}\label{cor:comparison1}
The operator $\Kop$ is nondecreasing and
\begin{equation}\label{OmegaB}
\Kop:\R{N}_+\to Q^-,
\end{equation}
where $Q^-$ is defined by \eqref{omme}.
\end{corollary}

\begin{proof}
The nondecreasing property is by Proposition~\ref{pro:comparison} and (\ref{OmegaB}) follows from (\ref{Omega}).
\end{proof}

\begin{definition}
The equation
\begin{align}\label{char}
\Kop \rho=\rho.
\end{align}
is said to be the \textit{characteristic equation} for the problem (\ref{rhoind1}). A nonnegative solution $\rho$ of (\ref{char}) is called a \textit{stationary solution} of (\ref{rhoind1}).
\end{definition}

 The set of stationary solutions is nonempty because $\rho=0$ is a (trivial) stationary solution. In section~\ref{sec:max} we characterize all nontrivial stationary solutions.

 As it was noticed before, the characteristic equation describes the possible scenario of the limit at infinity of solutions to (\ref{rhoind1}). The next lemma makes this observation more precise. First let us note that the  limit
$$
\rho_\infty:=\rho(M)\equiv \lim_{t\to \infty}\rho(t).
$$
is well-defined for any $\rho\in S_M$, where
$$
S_M:=\{\rho:\R{}_+\to \R{N}_+ \text{ such that  } \rho(t) \text{ is constant for  }t\ge M\}.
$$

\begin{lemma}
\label{lem:shift}
For any $f\in C(\R{}_+,\R{N}_+)$, $$\mathscr{L}_\mathbf{f}:S_M\to S_{M+A_m}
$$
and for any $\rho\in S_M$
\begin{equation}\label{koop}
(\mathscr{L}_\mathbf{f}\rho)_\infty=\Kop \rho_\infty.
\end{equation}
\end{lemma}

\begin{proof}
It follows from (\ref{FM}) and \ref{Hain4} that for any $t\ge M+A_m$ there holds
\begin{align*}
\mathscr{L}_\mathbf{f}\rho(t)=\mathcal{K}\rho(t)&=
\int_0^t\mathbf{m}(a)\Phi(a; \rho,t-a)\,da=\int_0^{A_m}\mathbf{m}(a)\Phi(a; \rho,t-a)\,da.
\end{align*}
Next, by virtue  our choice of $t$ we have for any $0\le a\le A_m$ that $t-a\ge t-A_m\ge M$, therefore $\Phi(a; \rho,t-a)=  \Phi(a;\rho_\infty,M)=\varphi(a;\rho_\infty)$. Therefore for all $t\ge M+A_m$
\begin{align*}
\mathscr{L}_\mathbf{f}\rho(t)=\int_0^{A_m}\mathbf{m}(a)\Phi(a;\rho_\infty,M) \,da\equiv
\int_0^{\infty}\mathbf{m}(a)\varphi(a; \rho_\infty)\,da=\Kop \rho_\infty
\end{align*}
which yields the desired conclusions.
\end{proof}

\subsection{The maximal solution of the characteristic equation}\label{sec:max}
A vector $\rho \in \R{N}_+$ is called  an \textit{upper} (resp. \textit{lower}) solution to equation (\ref{char}) if $\rho\geq\Kop \rho$ (resp. $\rho\leq\Kop \rho $).

\begin{lemma}
\label{lemupp}
The set of  lower solutions of \eqref{char} is bounded:
$$
\{\rho:\Kop \rho\le \rho\}\subset Q^-.
$$
Furthermore, any $\rho\in Q^+$ is an upper solution of \eqref{char}.
\end{lemma}

\begin{proof}
Indeed,  if $\rho\le \Kop \rho$ then applying (\ref{char}), (\ref{estphi}) and (\ref{omegadef}) one obtains
$$
\rho\le \int_0^{\infty}\mathbf{m}(a)\varphi(a;\rho)\,da \le \int_{a_m}^{A_m}\mathbf{m}(a)\mathbf{1}_N\frac{\omega_1}{a^{1/\gamma}}\,da \leq \omega_2\mathbf{1}_N
$$
which yields $\rho\in  Q^-$, and therefore the first claimed inclusion. Next, arguing similarly we have for any   $\rho\in  Q^+$  that
$$
\rho\ge \omega_2\cdot \textbf{1}_{N}\ge \int_0^{\infty}\mathbf{m}(a)\varphi(a;\rho)\,da=\Kop \rho
$$
which proves that $\rho$ is an upper solution of \eqref{char}.
\end{proof}

\begin{proposition}\label{it}
For any  $\rho^+\in Q^+$ the  limit
\begin{align}\label{limrho}
\theta:=\lim_{i\to \infty }\Kop ^i\rho^+
\end{align}
exists and $\theta$ is a solution of the characteristic equation. Furthermore,
\begin{enumerate}[label=$(\mathrm{{\roman*}})$]
\item \label{itemK2}
$\theta$  does not depend on a particular choice of $\rho^+\in Q^+$;
\item \label{itemK3}
if $\rho$ is an arbitrary lower solution of \eqref{char} then $\rho\le \theta$.
\end{enumerate}
\end{proposition}

\begin{proof}
By Lemma~\ref{lemupp},  $\Kop \rho^+\le \rho^+$. Thus, by the monotonicity of $\Kop$ we have for all $i\ge0$ that
$$
\Kop^{i+1}\rho\equiv \Kop^i\Kop \rho^+\le \Kop^i \rho^+,
$$
thus $\{\Kop^i \rho^+\}$ is a non-increasing sequence bounded from below: $\Kop^i \rho^+\ge0$. This implies the existence of the limit in (\ref{limrho}). Let us for a moment denote the limit by $\theta(\rho^+)$. It follows trivially that $\Kop \theta(\rho^+)=\theta(\rho^+)$. This proves that $\theta(\rho^+)$  is a solution of the characteristic equation. Next, let $\rho$ be an arbitrary lower solution of \eqref{char}. Then by Lemma~\ref{lemupp}
$$
\rho \le \Kop \rho\le \omega_2 \textbf{1}_{N}\le \rho^+.
$$
Iterating the latter inequality yields $\rho \le \Kop^i \rho\le \Kop^i \rho^+$, and passing to the limit as $i\to \infty$ we get $\rho\le \rho^+(\theta)$. This proves \ref{itemK3}. Now suppose that $\rho^+_1\in Q^+$. Then $\theta(\rho^+_1)$ is a solution of the characteristic equation, hence by \ref{itemK3}
$$
\theta(\rho^+_1)\le \theta(\rho^+),
$$
which, by symmetry, yields the equality in the latter inequality. This establishes the independence  of $\theta(\rho^+)$ on a choice of $\rho^+$, implying \ref{itemK2}.
\end{proof}

\begin{definition}
The unique $\theta$ defined by (\ref{limrho}) is called the \textit{maximal solution} of the characteristic equation.
\end{definition}

Note that the maximal solution $\theta$ does not depend on the initial population distribution $\mathbf{f}(a)$ and it is essentially determined by the maternity function $\mathbf{m}(a)$. As we shall see, the maximal solution  plays a distinguished role in the asymptotic analysis.


%
%
%

\subsection{The net reproductive rate dichotomy}

Throughout this section we  assume additionally that the condition \ref{Hain6} is also fulfilled.  Let us consider the scaled version  of $\Kop$ by
\begin{equation}\label{KR}
\mathscr{R}_\lambda x=\half{1}{\lambda} \Kop \lambda x, \quad x\in \R{N}_+, \qquad \lambda\in (0,\infty).
\end{equation}
Equivalently, we have component-wise
\begin{align}\label{Req}
\mathscr{R}_\lambda x:=\int_0^{\infty}\mathbf{m}(a)\mathbf{Y}(a; x,\lambda)\,da,
\end{align}
where
\begin{align*}
\mathbf{Y}(a; x , \lambda)&=\half{1}{\lambda}\varphi(a;\lambda x), \quad x\in \R{N}_+.
\end{align*}

Thus, the existence of a nontrivial solution to the  characteristic equation  (\ref{char}) is equivalent to the existence of a pair $(e,\lambda)$ , where a unit vector (a direction) $e\in\R{N}_+$, $\|e\|=1$ and a scalar $\lambda>0$ are such that
\begin{align}\label{Yk}
 e =\mathscr{R}_\lambda e.
\end{align}
The next lemma establishes that for each direction $e\in \R{N}_+$ there is at most one such pair.

\begin{lemma}
\label{lem:Ymono}
The operator $\mathscr{R}_\lambda$ is decreasing with respect to $\lambda$:
\begin{equation}\label{Ymono}
\lambda_2>\lambda_1\ge0 \quad \Rightarrow \quad \mathscr{R}_{\lambda_1}x\gg \mathscr{R}_{\lambda_2}x\quad \forall x\in \R{N}_+.
\end{equation}
In particular,
given an arbitrary direction $e\in\R{N}_+$,  $\|e\|=1$,
$$
\mathrm{card}\{\lambda>0:\,\lambda  e\in \Char\}\equiv
\mathrm{card}\{\lambda>0:\, e=\mathscr{R}_{\lambda} e\}\le 1.
$$
\end{lemma}

\begin{proof}
Since $\alpha=\lambda_2/\lambda_1>1$ we have from (\ref{conv2})
$$
\varphi(a;\lambda_2 x)=\varphi(a;\alpha \lambda_1 x)\le \alpha \varphi(a; \lambda_1 x),
$$
i.e. $Y(a;x, \lambda_2)\le Y(a;x, \lambda_1)$. This yields the weaker inequality $\mathscr{R}_{\lambda_1}x\ge \mathscr{R}_{\lambda_2}x$ for any $x\in \R{N}_+$. Next, by \ref{Hain6}  for an arbitrary $1\le k\le N$, there exists $\beta_k\le \sup\supp m_k$ such that the patch $k$ is accessible at $\beta_k$. By (\ref{conv3}), $\varphi_k(a;\alpha \lambda_1 x)< \alpha \varphi(a; \lambda_1 x)$ holds for any $a>\beta_k$. Thus, $Y_k(a;x, \lambda_2)< Y_k(a;x, \lambda_1)$ for $a>\beta_k$. Since $\supp m_k(a)\cap (\beta_k,\infty)$ has an nonempty interior, it follows from (\ref{Req}) that  $(\mathscr{R}_{\lambda_1}x)_k> (\mathscr{R}_{\lambda_2}x)_k$ for any $x\in \R{N}_+$. By the arbitrariness of $k$ one has (\ref{Ymono}). Next,  $e\in \R{N}_+$, $\|e\|=1$ be such that the set $\{\lambda>0:\, e=\mathscr{R}_{\lambda} e\}$ is nonempty,  say $e=\mathscr{R}_{\lambda_0} e$ for some $\lambda_0>0$. Then (\ref{Ymono}) yields
$$
\mathscr{R}_{\lambda_2} e\ll e=\mathscr{R}_{\lambda_0} e \ll\mathscr{R}_{\lambda_1} e
$$
for any $\lambda_1<\lambda_0<\lambda_2$. This proves that $\lambda_0$ is the only solution of $e=\mathscr{R}_{\lambda} e$.
\end{proof}

In the course of the proof of the lemma we  have established the following property.

\begin{corollary}
\label{cor:cf}
For any $0<\lambda<1$ and any $x\in \R{N}_+$ there holds $\lambda \varphi(a;x)\le  \varphi(a; \lambda x$).
\end{corollary}

%

The limit case $\lambda=0$ plays a distinguished role in the further analysis. Notice that
$Y_k(a;x,\lambda)$ is non-decreasing in $\lambda>0$ and by (\ref{Phibound}) $Y_k(a;x,\lambda)\le e^{N\|\mathbf{D}\|b}$, where the constant $b$ is from \ref{Hain1}. This implies that the limit
$$
Y_k(a;x):=\lim_{\lambda\to+0}Y_k(a;x,\lambda)
$$
does exist for any fixed $x\in \R{N}_+$, and the standard  argument shows that $\mathbf{Y}(a;x)$ is the unique solution of the \textit{linear} system
\begin{equation}
\left\{\begin{aligned}\label{Yax0}
\frac{d \mathbf{Y}(a;x)}{da}&=(\mathbf{D}(a)-\mathbf{M}(0,a))\mathbf{Y}(a;x),
\\
\mathbf{Y}(0;x)&=x.
\end{aligned}
\right.\end{equation}
Here
$$
\mathbf{M}(0,a)=\diag(\mu_1(a),\ldots, \mu_N(a))
$$
with $\mu_k(x)$ is defined by (\ref{uslovM3}).  Since $m_k\ge0$, the limit
\begin{equation}\label{reprmat}
\mathscr{R}_{0}x=\lim_{\lambda\to +0}
 \int_0^{\infty}\mathbf{m}(a)\mathbf{Y}(a;x,\lambda)\,da= \int_0^{\infty}\mathbf{m}(a)\mathbf{Y}(a;x)\,da.
\end{equation}
is well defined for each $x\in \R{N}_+$.

To proceed, we recall some standard concepts of the nonnegative matrix theory. A matrix $A$ is called {reducible} \cite{MincBook} if for some permutation matrix $P$
$$
P A P^t= \left(
    \begin{array}{ccc}
      A_{11} & 0 \\
      A_{21} & A_{22} \\
          \end{array}
  \right),
$$
where $A_{11}, A_{22}$ are square matrices, otherwise $A$ is called {irreducible}. There is  the following combinatorial characterization of the irreducibility, see \cite[p.~27]{Berman}, \cite[p.~671]{Meyer}: the condition that a \textit{nonnegative} matrix $A$ of order $n\ge 2$ is irreducible is equivalent to any of the following conditions:

\begin{enumerate}
  \item[(a)] no nonnegative eigenvector of $A$ has a zero coordinate;
  \item[(b)] $A$ has exactly one (up to scalar multiplication) nonnegative eigenvector, and this vector is positive;
  \item[(c)] $\alpha x\ge Ax$ and $x>0$ implies $x\gg0$;
  \item[(d)] the associated graph $\Gamma(A)$ is strongly connected.
  \end{enumerate}

\begin{lemma}
The map $\mathscr{R}_{0}: \R{N}\to \R{N}$  defined by $(\ref{reprmat})$ is linear and strongly positive, i.e. $x>0$ implies $\mathscr{R}_{0}x\gg 0$. In particular, $\mathscr{R}_{0}$ is an irreducible matrix. Furthermore,
\begin{equation}\label{Ymono0}
\mathscr{R}_{\lambda}x \ll\mathscr{R}_{0} x, \quad \forall x\in \R{n}_+, \,\,\lambda>0.
\end{equation}
 \end{lemma}

 \begin{proof}
Indeed, the linearity follows immediately by (\ref{reprmat}) and (\ref{Yax0}). Since the matrix  $\mathbf{M}(0,a)$ is diagonal, the associated digraphs of the matrices $\mathbf{D}(a)$ and $\mathbf{D}(a)-\mathbf{M}(0,a)$  are equal. Therefore, using \ref{Hain6} readily yields  that $Y_k(a;x,0)> 0$ for any $a>\beta_k$. Hence, repeating the argument of Lemma~\ref{lem:Ymono} we have from (\ref{reprmat}) and \ref{Hain4} that $(\mathscr{R}_{0}x)_k>0$ for any $k$. This proves $\mathscr{R}_{0}x\gg 0$. Suppose by contradiction  that $\mathscr{R}_{0}$ is reducible. Then for some permutation matrix $P$
\begin{equation}\label{block}
P \mathscr{R}_{0} P^t= \left(
    \begin{array}{ccc}
      A_{11} & 0 \\
      A_{21} & A_{22} \\
          \end{array}
  \right),
\end{equation}
where $A_{11}, A_{22}$ are square matrices. Let $x>0$ be a vector in $\R{N}_+$ with all first $m$ coordinates zero, where $m$ is the order of $A_{11}$. By (\ref{block}) $P \mathscr{R}_{0} P^t x$ has the same property, i.e. the vector $\mathscr{R}_{0}P^t x$ has at least $m$ zero coordinates which  contradicts to the fact that  $\mathscr{R}_{0}P^t x\gg0$. This proves the irreducibility. Finally, (\ref{Ymono0})  follows from (\ref{Ymono}).
 \end{proof}

 \begin{corollary}
If $\mathscr{R}_{0} e\le e$ for any $e\in\R{N}_+$, $\|e\|=1$, then the characteristic equation $(\ref{char})$ admits only trivial solutions.
\end{corollary}

\begin{proof}
Indeed, if $\rho\ne0$ is a nontrivial solution of (\ref{char}) then by  (\ref{KR})  $e=\rho/\|\rho\|$ is a solution of $\mathscr{R}_{\lambda}e =e$ for $\lambda=\|\rho\|$. On the other hand, using the assumption and (\ref{Ymono0}) we obtain
$$
e=\mathscr{R}_{\lambda}e \ll\mathscr{R}_{0} e\le e,
$$
a contradiction follows.
\end{proof}

%
%

Let us denote by $\sigma(\mathscr{R}_{0})$ the spectral radius of the linear map $\mathscr{R}_{0}$. Combining the irreducibility of $\mathscr{R}_{0}$ with the Perron-Frobenius theorem \cite[Theorem~1.3.26]{Berman} implies the following important observation.

\begin{corollary}\label{cor:Krein}
The spectral radius $\sigma(\mathscr{R}_{0})>0$ and it is a simple eigenvalue of $\mathscr{R}_{0}$. If $x$ is an eigenvector of $\mathscr{R}_{0}$ then $x\gg 0$. If  $\lambda\ne\sigma$ is another eigenvalue of $\mathscr{R}_{0}$ then $|\lambda|<\sigma$. Furthermore, the  Collatz-Wielandt identity holds
$$
\max_{x>0} \min_{\substack{1\le i\le N\\x_i\ne 0}} \frac{(\mathscr{R}_{0}x)_i}{x_i}=\sigma(\mathscr{R}_{0}).
$$
\end{corollary}

\begin{definition}
The linear map $\mathscr{R}_{0}$ is called the \textit{net reproductive map} associated to the problem (\ref{sys1}). Its spectral radius $\sigma(\mathscr{R}_{0})$ is called  the \textit{net reproductive rate}.
 \end{definition}

The latter definition can be motivated as folows. For a single patch model, i.e.  $N=1$, the linear system (\ref{Yax0}) becomes a single equation
$$
\frac{d}{da}Y_1(a;x,0)=-\mu(a)Y_1(a;x,0),
 $$
with an explicit solution $Y_1(a;x,0)=x\exp(-\int_{0}^a \mu(s)ds)$. Thus (\ref{reprmat}) yields
\begin{align}\label{rateR0}
\mathscr{R}_{0}x&=R_0x,
\end{align}
where
\begin{align}\label{sigmaind}
\sigma(\mathscr{R}_{0})=R_0&=\int_0^{\infty}m(a)e^{-\int_0^a\mu(s)ds}\,da.
\end{align}
The quantity $R_0$ is well-established and is known as
the (inherent) {net reproductive rate}  in the linear time-independent model on a single patch \cite{Iannelli95}, \cite{Cushing98}; see also \cite{we1} or \cite{we2}. Note that in this case,
\begin{align}\label{pimort}
\Pi(a)=e^{-\int_0^a\mu(s)ds}
\end{align}
is the survival probability, i.e. the probability for an individual to survive to age $v$. Then $R_0$ is the expected number of offsprings per individual per lifetime. Recall that in the one-dimensional case, $R_0$ is related to the intrinsic growth rate of population by the characteristic equation. Namely,  when $R_0>1$ population is growing, while for $R_0\leq 1$ population is declining.

The next result extends this dichotomy onto the general multipatch case. Recall that

\begin{theorem}[The Net Reproductive Rate Dichotomy]\label{th:main}
If $\sigma(\mathscr{R}_{0})\leq 1$ then $\theta= 0$ and the equation $(\ref{char})$ has no nontrivial solutions. If $\sigma(\mathscr{R}_{0})> 1$ then $\theta \gg 0$ and $\theta $ is the only nontrivial solution of the characteristic equation $(\ref{char})$.
\end{theorem}

\begin{proof}
First let us assume that $\sigma(\mathscr{R}_{0})\leq 1$  and suppose by contradiction that $\Kop \rho=\rho$ for some $\rho>0$. Let $\lambda=\|\rho\|$ and $e=\rho/\lambda$, then by (\ref{KR}) and (\ref{Ymono0}),
$$
\Rol e\gg\mathscr{R}_\lambda e=\half{1}{\lambda} \Kop \lambda e=\half{1}{\lambda}\Kop \rho=\half{1}{\lambda}\rho=e.
$$
The latter easily implies that there exists $t>1$ such that $\Rol e\ge te$. On iterating the obtained inequality yields $\Rol^k e\ge t^ke$, thus
$$
\sigma(\Rol)=\lim_{k\to \infty}\|\Rol^k\|^{1/k}\ge t>1,
$$
a contradiction.

Now suppose that $\sigma(\mathscr{R}_{0}) >1$. By Corollary~\ref{cor:Krein}, there exists a positive eigenvector  $e_0\gg0$ of $\mathscr{R}_{0}$.
Since $e_0\gg0$ there exists $\lambda>0$ such that $\lambda e_0\ge \skob{\omega_2 }$, where $\omega_2 $ is defined by (\ref{omegadef}).  By (\ref{omme}), $\rho^+:=\lambda e_0\in  Q^+$, hence Lemma ~\ref{it} implies that
$$
\theta=\lim_{i\to \infty }\Kop ^i\rho^+\in \Char
$$
is a  solution to (\ref{char}). On the other hand, since $\sigma(\mathscr{R}_{0}) >1$ we have
$$
\mathscr{R}_{0}e_0=\sigma(\mathscr{R}_{0})e_0\gg e_0.
$$
hence, by the continuity argument for some $\lambda>0$ small enough there holds
$$
\mathscr{R}_{\lambda}e_0\gg e_0.
$$
Therefore, setting $\rho^-:=\lambda e_0$ we obtain
$$
\Kop \rho^-=\Kop\lambda e_0=\lambda \mathscr{R}_{\lambda}e_0\gg \lambda e_0=\rho^-,
$$
i.e. $\rho^-$ is an lower solution of (\ref{char}). In other words, $\rho^-\in \Char^{\mathrm{low}}$, thus  \ref{itemK3} of Proposition~\ref{it} yields
$$
\theta\ge \rho^-\gg0,
$$
thus $\theta$ is a nontrivial solution.

In order to establish the uniqueness of a nontrivial solution (i.e. that $\mathrm{card} (\Char)=1$), we will follow the idea of Krasnoselskii and Zabreiko from \cite[Ch.~6]{KrZa}. To this end, let us suppose that $\theta_1,\theta_2$ be two nontrivial solutions to (\ref{char}). Then $\theta_1,\theta_2\gg0$. If $\theta_1\ne \theta_2$ then at least one of inequalities $\theta_1\leq \theta_2$ and $\theta_2\leq \theta_1$ is not valid. Suppose that $\theta_1\leq \theta_2$ is \textit{not} satisfied. Since $\theta_1\gg 0=0\cdot \theta_2$, the set $\{\lambda\ge0: \theta_1\geq \lambda\cdot  \theta_2\}$ is non-empty and the following supremum is well-defined
$$
\lambda_0=\sup\{\lambda\ge0: \theta_1\geq \lambda \theta_2\}.
$$
Since $\theta_1\gg 0$ there exists $\epsilon>0$ such that $\theta_1\ge \epsilon \theta_2$, hence $\lambda_0\ge \epsilon>0$. On the other hand, by the assumption $\theta_1\not \leq \theta_2$, therefore  we also have $1\not\in \{\lambda\ge0: \theta_1\geq \lambda \theta_2\}$, thus $\lambda_0\i\mathbf{n}(0,1)$. By the continuity,
$ \theta_1\ge \lambda_0 \theta_2,$
by the monotonicity of  $\Kop$ and $\lambda_0<1$ one has
\begin{align*}
\theta_1&=\Kop \theta_1 \geq \Kop(\lambda_0 \theta_2) =\lambda_0\mathscr{R}_{\lambda_0}(\theta_2)\gg
\lambda_0\mathscr{R}_{1}(\theta_2)\\
&=\lambda_0\Kop \theta_2=\lambda_0 \theta_2,
\end{align*}
Thus, $\theta_1\gg\lambda_0 \theta_2$, implying $\theta_1\ge(\delta+\lambda_0) \theta_2$ for some small positive $\delta$. The latter inequality contradicts the definition of $\lambda_0$. This finishes the proof of the uniqueness.
%
%
%
\end{proof}

\subsection{Asymptotic behaviour of a general solution of (\ref{rhoind1})}
Let us return to the general equation (\ref{rhoind1}). If the initial distribution of population vanishes: $\mathbf{n}(a,0)=\mathbf{f}(a)=0$, the uniqueness of solution of \eqref{sys1} immediately implies that the population density  $\mathbf{n}(a,t)\equiv 0$ for all $a,t\ge0$. This conclusion also holds true even under a weaker assumption that $\mathcal{F}\mathbf{f}\equiv 0$. The latter is evident from the biological point of view: the population disappears if its initial distribution is older that the maternity period. Taking into account these observations, it is naturally to assume that
\begin{equation}\label{nonF}
\mathcal{F}\mathbf{f}\not\equiv 0.
\end{equation}
The main result of this section states that  under this assumption, any solution of (\ref{rhoind1}) behaves asymptotically as the maximal solution.

\begin{theorem}\label{th:est}
Let $\chi$ be the   solution to $(\ref{rhoind1})$ satisfying \eqref{nonF}. Then
\begin{equation}\label{dichot}
\lim_{t\to\infty}\chi(t)=\theta.
\end{equation}
\end{theorem}

We start with two results describing the upper and lower solutions to equation (\ref{rhoind1}).

\begin{lemma}\label{lemmaupp}
Let $\chi$ be a solution to $(\ref{rhoind1})$. Then
\begin{equation}\label{limsup}
\limsup_{t\rightarrow\infty}\chi(t)\leq \theta,
\end{equation}
where the latter inequality should be understood component-wise.
\end{lemma}

\begin{proof}
Let  $\rho^+$ be an arbitrary stationary upper solution to (\ref{rhoind1}), i.e.
\begin{equation}\label{uppK}
\rho^+\ge\mathscr{L}_\mathbf{f}\rho^+.
\end{equation}
Notice that that the class of stationary upper solutions is nonempty. Indeed, it follows from (\ref{Omega}) that, for example, $2(\omega_2 +\epsilon)\mathbf{1}$ is such a an upper solution for any $\epsilon>0$.  Now, let us define the iterative sequence  by
\begin{align*}
\rho^{(i)}&=\Kop^{i+1} \rho^+ \text{ for }i\ge0 \text{ and } \rho^{(0)}=\rho^+,\\
\chi^{(i)}&=\mathscr{L}_\mathbf{f}^{i+1} \rho^+ \text{ for }i\ge0 \text{ and } \chi^{(0)}=\rho^+.
\end{align*}
Then applying the argument of the proof of Proposition~\ref{it} yields  that $\{\rho^{(i)}\}$ is non-increasing:
$$
\rho^{(i+1)}\leq\rho^{(i)}, \quad \forall i\ge0.
$$
Also, since $\rho^{(i)}$ is a constant vector function, it follows by Lemma~\ref{lem:shift} that $\mathscr{L}_\mathbf{f} \rho^{(i)}\in S_{A_m}$ and also that
\begin{equation*}
\mathscr{L}_\mathbf{f} \rho^{(i)}(t) \equiv \Kop\rho^{(i)}(t),\quad \forall t\ge A_m.
\end{equation*}

We claim that for any $j\ge0$
\begin{enumerate}
\item[(a)] $\chi^{(j+1)}\le \chi^{(j)}$ for all $t\ge0$;
  \item[(b)] $\chi^{(j)}= \rho^{(j)}$ for $t\ge jA_m$.
\end{enumerate}
The proof is by induction. Notice that (b) holds trivially for  $j=0$, and  by the assumption (\ref{uppK})
$$
\chi^{(1)}=\mathscr{L}_\mathbf{f}\chi^{(0)}=
\mathscr{L}_\mathbf{f}\rho^{+}\le \rho^+=\chi^{(0)}
$$
which yields (a) for $j=0$.  Let the claims (a)--(b) hold true for some $j\ge 1$. Then  (a) follows from the monotonicity of $\mathscr{L}_\mathbf{f}$:
$$
\chi^{(j+1)}=\mathscr{L}_\mathbf{f}\chi^{(j)}\le
\mathscr{L}_\mathbf{f}\chi^{(j-1)}=\chi^{(j)}.
$$
Furthermore by the assumption $\chi^{(j)}\in S_{jA_m}$ and $\chi^{(j)}_\infty= \rho^{(j)}$. Hence Lemma~\ref{lem:shift} yields $$\chi^{(j+1)}=\mathscr{L}_\mathbf{f}\chi^{(j)}\in S_{(j+1)A_m}
$$
and
$$
\chi^{(j+1)}_\infty=(\mathscr{L}_\mathbf{f}\chi^{(j)})_\infty=\Kop \chi^{(j)}_\infty =
\Kop \rho^{(j)}=\rho^{(j+1)},
$$
which yields (b) for $j+1$.

Next, it follows from (a) and the boundedness of the image of $\mathcal{L}$ that $\{\chi^{(j)}\}$ is non-increasing and   bounded from below, thus has a limit which obviously is a solution of (\ref{rhoind1}). By the uniqueness,  $\lim_{j\to\infty}\chi^{(j)}(t)=\chi(t)$.  Now, let $1\le k\le N$. Then the sequence of the coordinate functions $\chi_k^{(j)}(t)$ is non-increasing  with respect to $j$ and $\lim_{j\to\infty}\chi_k^{(j)}(t)=\chi_k(t)$. Let $\epsilon>0$. Since $\lim_{j\to\infty}\rho_k^{(j)}=\theta_k$, there exists $j_0$ such that  $\theta_k\le \rho_k^{(j)}\le \theta_k+\epsilon$ for all $j\ge j_0$. This implies that $\chi_k^{(j)}(t)\le \theta_k+\epsilon$ for all $j\ge j_0$ and $t\ge jA_m$. Passing to the limit $j\to \infty$ we obtain $\chi_k(t)\le \theta_k+\epsilon$ for $t\ge jA_m$ which easily implies (\ref{limsup}).

\end{proof}

\begin{lemma}\label{lemmalow}
Let $\chi$ be a solution to  \eqref{rhoind1}. If there exists a lower solution $\rho^-$ to \eqref{rhoind1}, i.e. $\mathscr{L}_\mathbf{f}\rho^-\ge \rho^-$ such that $\rho^-\in S_M $ for some $M\ge0$ and $\rho^-_\infty\ne0$
then $\lim_{t\rightarrow\infty}\chi(t)=\theta$.
\end{lemma}

\begin{proof}
As above, let us consider the sequence of iterations
\begin{align*}
\chi^{(j)}&=\mathscr{L}_\mathbf{f}^j\chi^{(0)} \text{ for }i\ge0 \text{ and }  \chi^{(0)}(t)=\rho^{-},\\
\rho^{(j)}&=\Kop^j\rho^{(0)} \text{ for }i\ge0 \text{ and } \rho^{(0)}(t)=(\rho^{-})_\infty,
\end{align*}
By Lemma~\ref{lem:shift}, $\chi^{(j)}\in S_{M+jA_m}$. Furthermore, by (\ref{koop})
$$
\chi^{(1)}_\infty=(\mathscr{L}_\mathbf{f}\rho^-)_\infty=\Kop \rho^{-}_\infty=\rho^{(1)}.
$$
Using an induction argument readily yields
\begin{equation}\label{rhochi}
\chi^{(j)}_\infty=\rho^{(j)}, \quad \forall j\ge0.
\end{equation}

Since $\mathscr{L}_\mathbf{f}\rho^-\ge \rho^-$, we have $\chi^{(1)}\ge \chi^{(0)}$, thus by the monotonicity of $\mathscr{L}_\mathbf{f}$, $\chi^{(j+1)}\ge \chi^{(j)}$. This proves that $\{\chi^{(j)}(t)\}$ is a  nondecreasing sequence. Furthermore, (\ref{rhochi}) implies that
$$
\rho^{(j+1)}=\chi^{(j+1)}_\infty\ge \chi^{(j)}_\infty=\rho^{(j)},
$$
thus, $\{\rho^{(j)}\}$ is also a  nondecreasing sequence. Furthermore, since $\rho^-_\infty\ne0$, we have that $\rho^{(j)}\gg 0$ for $j\ge1$. By Lemma~\ref{lemOmega} the both sequences are bounded from above by $\omega_2 \mathbf{1}_N$. Thus, the limits $\rho:=\lim_{j\to \infty}\rho^{(j)}$ and $\bar\chi:=\lim_{j\to \infty}\chi^{(j)}(t)$  exist and  solve $\Kop\rho=\rho$ and   $\mathscr{L}_\mathbf{f}\bar \chi=\bar \chi$, respectively, where $\rho\gg0$. By the corresponding  uniqueness results, we have $\rho=\theta$ and $\bar\chi=\chi$. Arguing as in Lemma~\ref{lemmaupp}, we obtain $\liminf_{t\rightarrow\infty}\chi(t)\geq \theta$ (the latter is understood component-wise). Hence (\ref{limsup}) implies the existence of the limit $\lim_{t\rightarrow\infty}\chi(t)=\theta$.
\end{proof}

\begin{proof}[Proof of Theorem \ref{th:est}]

If $\sigma(\Rol)\leq 1$, then Theorem \ref{th:main} yields $\theta \equiv 0$, then (\ref{limsup}) immediately yields (\ref{dichot}). Therefore we shall suppose that $\sigma(\Rol)>1$.  Let $\chi$ be the unique solution to (\ref{rhoind1}) and let $\theta\gg0$ be the unique maximal solution  of (\ref{char}). By Lemma~\ref{lemmalow}, it suffices to show that there exists a lower solution $\rho^-$ to (\ref{rhoind}) such that $\rho^-\in S_M $ for some $M\ge0$ and $\rho^-_\infty\ne0$. In the remained part of the proof we shall construct such a solution. Let us consider an auxiliary sequence of iterations
\begin{align*}
\rho^{(j)}=\mathscr{L}_\mathbf{f}\rho^{(0)}  \text{ for } j\ge1 \text{ and } \rho^{(0)}\equiv 0.
\end{align*}
We claim that the new function $\rho^-(t)$ defined by
\begin{equation}\label{rho-}
\rho^-(t)=
\left\{
  \begin{array}{ll}
    \rho^{(j)}(t), & 0\leq t\leq M, \\
    \lambda\theta , &  t>M,
  \end{array}
\right.
\end{equation}
is a lower solution to equation (\ref{rhoind1}) for certain $M>A_m$,  sufficiently large $j\ge1$ and sufficiently small  $\lambda>0$ to be specified later. To this end, first notice that
$$
\rho^{(1)}=\varphi:=\mathcal{F} \mathbf{f} \ge 0=\rho^{(0)},
$$
hence using an induction by $j\ge1$, one gets
$$
\rho^{(j+1)}=\mathscr{L}_\mathbf{f}\rho^{(j)} \ge \mathscr{L}_\mathbf{f}\rho^{(j-1)} =\rho^{(j)},
$$
i.e. the sequence $\rho^{(j)}$ is non-decreasing in $j$. It also follows from the latter inequality that $\rho^{(j)}\le \mathscr{L}_\mathbf{f}\rho^{(j)} $, i.e. $\rho^{(j)}$ is a lower solution to   (\ref{rhoind1}). Hence, $\rho^{-}(t)$ defined by (\ref{rho-}) is a lower solution to (\ref{rhoind}) in the interval $t\in[0,M]$. In particular,
\begin{equation*}
(\mathscr{L}_\mathbf{f}\rho^-)(t)-\rho^-(t)\ge 0\text{ for }t\in[0,M].
\end{equation*}

 Next, we assume that  $t\in[M,M+A_m]$. By the assumption $M>A_m$, hence one has $(\mathcal{F}\mathbf{f})(t)=0$ and $\mathscr{L}_\mathbf{f}\rho^-=\mathcal{K}\rho^-$. We have by (\ref{rho-}) and condition~\ref{Hain4} that
\begin{align*}
\mathcal{K}\rho^-(t)&= \int_{0}^{A_m}\mathbf{m}(a)\Phi(a;\rho^{-},t-a)\,da  \\
&= \int_{0}^{t-M}\mathbf{m}(a)\varphi(a;\lambda \theta )\,da + \int_{t-M}^{A_m}\mathbf{m}(a)\Phi(a;\rho^{(j)},t-a)\,da
\end{align*}
On the other hand, since $\Kop\theta=\theta $, we have
\begin{align*}
\theta &= \int_0^{A_m}\mathbf{m}(a)\varphi(a;\theta )\,da.
\end{align*}
This yields by virtue of $\rho^-(t)=\lambda\theta $ for $t\in(M,M+A_m)$ and (\ref{phiphi}) that
\begin{align}
(\mathscr{L}_\mathbf{f}\rho^--\rho^-)(t) &=(\mathcal{K}\rho^--\rho^-)(t)=(\mathcal{K}\rho^--\theta )(t) \nonumber\\
&= \int_{0}^{t-M}\mathbf{m}(a)(\varphi(a;\lambda\theta )-\lambda\varphi(a;\theta ))\,da \label{int51}\\
&\quad + \int_{t-M}^{A_m}\mathbf{m}(a)(\Phi(a;\rho^{(j)},t-a) -\lambda\Phi(a;\theta,t-a ))\,da. \label{int52}
\end{align}
We claim that the integrals (\ref{int51}) and (\ref{int52}) are nonnegative. The first integral is nonnegative by virtue of Corollary~\ref{cor:cf}.
To show that (\ref{int52}) is nonnegative, let us estimate function
$\Phi_k(a;\rho^{(j)},t-a)$ from below. By \ref{Hain6}, $m_k(a)\ge \delta>0$ for all $a\in[a_k,b_k]$, where $a_k\geq\beta_k$ and $b_k$ are the same as in Lemma~\ref{lem:conv}. Since ${\mathcal F}(\mathbf{f})$ is not identically zero, there exists an interval $[s_1,s-2]$, where this function is positive. Applying Lemma~\ref{lem:conv} for $\rho=\rho^{(1)}={\mathcal F}(\mathbf{f})$, we get that
$$
({\mathcal K}\rho^{(1)})_k(t)>0\;\;\;\mbox{for $t\in [s_1+a_k,s_2+b_k]$}.
$$
Therefore
$$
\rho^{(2)}(t)={\mathcal K}\rho^{(1)}(t)+{\mathcal F}\mathbf{f}(t)>0\,\;\;\mbox{for $t\in[s_1+a_k,s_2+b_k]$, $k=1,\ldots,N$,}
$$
and, in particular this is true for $k=1$. Repeating this argument yields
$$
\rho^{(j)}(t)>0\;\;\;\mbox{for $t\in[s_1+(j-1)a_1,s_2+(j-1)b_1]$}.
$$
This implies that
$$
\Phi_k(a;\rho^{(j)},t-a)>0\;\;\;\mbox{for $a\geq\beta_k$ and $t-a\in [s_1+(j-1)a_1,s_2+(j-1)b_1]$}.
$$
Now we choose the index $j$ and the number $M$ to satisfy
$$
[M-A_m,M+A_m]\subset [s_1+(j-1)a_1,s_2+(j-1)b_1].
$$
Then
\begin{equation}\label{Koz1}
\Phi_k(a;\rho^{(j)},t-a)>0\;\;\;\mbox{for $a\in[\beta_k,A_m]$ and $t\in[0,A_m]$}.
\end{equation}
Therefore,
$$
\Phi_k(a;\rho^{(j)},t-a)\leq\lambda\Phi_k(a;\theta,t-a)
$$
for such $a$ and $t$ if $\lambda$ is sufficiently small positive number. This gives positivity of (\ref{int52}) for
$t\geq M+a_k$. If $t\leq M+a_k$ then the first integral in (\ref{int52}) is estimated from below by
$$
\int_{a_k}^{b_k}m_k(a)\Phi_k(a;\rho^{(j)},t-a)da
$$
and it is positive for $t\in[M,M+A_m]$. Since the functions $\Phi_k$ are uniformly bounded this implies the positivity of (\ref{int52}) for $M\leq t\leq M+a_k$ when $\lambda$ is small.

Finally, if $t\geq M+A_m$, then since $\mathcal{F}\mathbf{f}(t)=0$  we have by virtue of Corollary~\ref{cor:cf} that
\begin{align*}
(\mathscr{L}_\mathbf{f}\rho^--\rho^-)_k(t)=(\mathcal{K}\rho^--\rho^-)_k(t) &= \int_{0}^{A_m}m_k(a)(\varphi_k(a;\lambda\theta ) -\lambda\varphi_k(a;\theta ))\,da  \ge0
\end{align*}
This proves that the function $\rho^-(t)$ defined by (\ref{rho-}) is a lower solution to equation (\ref{char}), therefore by Lemma \ref{lemmalow} we have the desired convergence that completes the proof.
\end{proof}

\subsection{Asymptotics of total population}
According to the assumption made in the beginning of this section, the maximal length of life is constant:  $B(t)\equiv b$. Then the total (multipatch) population $\mathbf{P}(t)$ at time $t$ is the vector-function
\begin{align}\label{total}
\mathbf{P}(t)=\int_0^{b}\mathbf{n}(a,t)\,da.
\end{align}
 Then we have the following result.

\begin{theorem}\label{th:estN}
Let $n(a,t)$ be the solution of \eqref{sys1} and let the condition \eqref{nonF} hold. Then the following dichotomy holds: if $\sigma(\Rol)\le1$ then $\mathbf{P}(t)\rightarrow 0$ as $t\rightarrow\infty$, and  if $\sigma(\Rol)>1$ then
\begin{equation}\label{assym}
\lim_{t\to\infty}\mathbf{P}(t)= \int_0^{b}\varphi(a;\theta )\,da,
\end{equation}
where $\theta $ is the maximal solution to the characteristic  equation.
\end{theorem}

\begin{proof}
Denote by $\rho(a)$ the newborns function determined by $\mathbf{f}(a)$ by virtue of (\ref{rhoind1}). We have  for general $t>0$
\begin{align*}
\mathbf{P}(t)&= \int_0^{\min\{t,b\}}\Phi(a;\rho,t-a)\,da+\int_{\min\{t,b\}}^{b}\Psi(a;\mathbf{f},a-t)da.
\end{align*}
On the other hand, by \ref{Hain5}  $\supp \mathbf{f}\subset [0,b]$, hence using (\ref{nPhiPsi}) we have for any $t>b$ that
\begin{align*}
\mathbf{P}(t)&= \int_0^{b}\Phi(a;\rho,t-a)\,da.
\end{align*}
Next, by Theorem~\ref{th:main} and Theorem \ref{th:est} we have $\lim_{t\to\infty}\rho(t)=\theta$ and furthermore by (\ref{balanceV}) there holds $h(a):=\Phi(a;\rho,t-a)$ satisfies
\begin{equation}\label{balanceV1}
\renewcommand\arraystretch{1.5}
\left\{
\begin{array}{rl}
  \frac{d }{da}{h}(a)&=-\mathbf{M}({h}(a),a){h}(a)+\mathbf{D}(a){h}(a), \\
  h(0)&= \rho(t-a),
\end{array}
\right.
\end{equation}
By continuity of solutions (\ref{balanceV1}) with respect to a parameter and (\ref{varphieq}), we have for any fixed $a>0$ that $$\lim_{t\to\infty}\Phi(a;\rho,t-a)=\varphi(a;\theta).
$$
This readily yields \eqref{assym}.
%
\end{proof}

\subsection{Estimates for the net reproductive rate and for the maximal solution}\label{sec:est}

In this section we shall assume that the condition $(\ref{cond:D})$ hold, i.e.
$$
\sum_{k=1}^ND_{kj}(a)\le 0, \quad 1\le j\le N.
$$
The biological meaning of the latter inequality is that individuals do not reproduce during migration (but can die). This condition  immediately implies that
$$
D_{kk}(a)\le0.
$$
Throughout this section, we use the following notation:
$$
m(a)=\max_{1\le k\le N}m_k(a),\qquad \mu(a):=\min_{1\le k\le N}\mu_k(a).
$$

\begin{proposition}\label{prop:estsigma}
Under the made assumptions,
\begin{align}\label{est:s}
\max_{1\le k\le N}\int_0^{\infty}\!\!m_k(a)e^{-\int_0^a(\mu_k(v)+|D_{kk}(v)|)dv}\,da \le \sigma(\Rol) \le \int_0^{\infty}\!\!m(a)e^{-\int_0^a\mu(v)dv}\,da.
\end{align}
\end{proposition}

\begin{proof}
By Corollary~\ref{cor:Krein} there exists an eigenvector $\rho\gg 0$ of $\Rol$ corresponding the maximal eigenvalue $\sigma(\Rol)$, i.e. $\Rol\rho=\sigma(\Rol)\rho$.  Let us consider the problem (\ref{Yax0})  with the initial condition $x=\rho$. Using the assumption (\ref{cond:D}) and summing up the  equations (\ref{Yax0})  for all $1\le k\le N$ we obtain that $\psi(a)=\sum_{k=1}^N Y_k(a;\rho)$ satisfies
\begin{equation*}
\renewcommand\arraystretch{1.5}
\left\{
\begin{array}{rcl}
\frac{d }{da}\sum_{k=1}^N\psi(a)&\le&-\mu(a)\psi(a),\\
\psi(0)&=&\sum_{k=1}^N\rho_k,
\end{array}
\right.
\end{equation*}
which readily yields
$$
\psi(a)\le  e^{-\int_o^a\mu(v)dv}\sum_{k=1}^N\rho_k.
$$
Then by (\ref{reprmat})
$$
\sigma(\mathscr{R}_{0})\sum_{k=1}^N\rho_k=\sum_{k=1}^N(\mathscr{R}_{0}x)_k\le  \int_0^{\infty}m(a)\psi(a)\,da\le \sum_{k=1}^N\rho_k\int_0^{\infty}m(a)e^{-\int_o^a\mu(v)dv}\,da.
$$
Since the sum $\sum_{k=1}^N\rho_k>0$ we arrive at the right hand side of (\ref{est:s}).

Now, in order to prove the left hand side inequality in (\ref{est:s}),
notice that in the made notation by virtue of $D_{kj}(a)\ge0$ for $k\ne j$ and $Y_j(a,\rho)\ge0$ for all admissible $a$ we have
$$
\frac{d}{da}Y_k(a;\rho) \ge d_{kk}(a)Y_k(a;\rho)=-(\mu_k(a)+|D_{kk}(a)|)Y_k(a;\rho),
$$
which yields in virtue of $Y_k(0,\rho)=\rho_k$ that
$$
Y_k(a;\rho)\ge \rho_ke^{-\int_0^a(\mu_k(v)+|D_{kk}(v)|)dv}.
$$
Combining this with (\ref{reprmat}) we obtain
$$
\sigma(\mathscr{R}_{0})\rho_k
=\int_0^{\infty}m_k(a)Y_k(a;\rho)\,da \ge \rho_k \int_0^{\infty}m_k(a)e^{-\int_0^a(\mu_k(v)+|D_{kk}(v)|)dv}\,da,
$$
thus implying (\ref{est:s}) by virtue of $\rho_k>0$.
\end{proof}

\begin{remark}
The estimates (\ref{est:s}) are optimal. Indeed, if $D_{kj}\equiv 0$, the system (\ref{Yax0}) splits into separate equations
$$
\frac{d}{da}Y_k(a;x)=-(\mu_k(a)+|D_{kk}(a)|)Y_k(a;x), \quad Y_k(0,x)=x_k, \quad 1\le k\le N,
$$
implying that each $e_k$ is an eigenvector of $\Rol$ with eigenvalue $$
\lambda_k=\int_{0}^\infty m_k(a)e^{-\int_{0}^a(\mu_k(a)+|D_{kk}(a)|)ds}da,
$$
therefore $\sigma(\Rol)=\max_{k}\lambda_k$ is exactly the left hand side of (\ref{est:s}). On the other hand, suppose all patches to have the same birth and death rates: $m_k(a)\equiv m(a)$ and $\mu_k(a)\equiv \mu(a)$ for any $1\le k\le N$, and also that the dispersion is absent: $D\equiv 0$. Then a similar argument yields $\sigma(\Rol)=\int_0^{\infty}\!\!m(a)e^{-\int_0^a\mu(v)dv}\,da$ implying the exactness of the upper estimate in (\ref{est:s}).
\end{remark}

In order to establish the corresponding estimates for the maximal solution $\theta$ we consider an auxiliary function
\begin{equation}\label{not20}
\begin{aligned}
\widetilde{M}(t,a):=\frac{1}{t}\min_{\xi \in S(t)}\sum_{i=1}^N \xi_iM_i(\xi_i,a),\quad t>0,
\end{aligned}
\end{equation}
where the minimum is taken over the simplex
$$
S(t)=\{\xi\in \R{N}_+: \sum_{i=1}^N\xi_i=t\}.
$$

\begin{lemma}
\label{lem:M}
In the above notation, $\widetilde{M}(t,a)$ is nondecreasing in $t>0$ and
\begin{equation}\label{meanM}
\lim_{t\to+0}\widetilde{M}(t,a)=\mu(a).
\end{equation}
Furthermore,
 \begin{equation}\label{mtilde}
\widetilde{M}(t,a)-\mu(a)\ge \frac{p(a)}{N^\gamma} t^\gamma,
\end{equation}
where $p(a)$ is the function from \ref{Hain2}.
\end{lemma}

\begin{proof}
If $0<t'\le t$ then  $\lambda=t'/t\le1$. If $\xi\in S(t)$ is the minimum point of (\ref{not20}) then  $\xi'=\lambda \xi\in S(t')$, hence using the monotonicity condition in \ref{Hain2} and $\xi\ge \xi'$ we obtain
$$
\widetilde{M}(t,a)=\frac{1}{t}\sum_{i=1}^N \xi_iM_i(\xi_i,a)=
\frac{1}{\lambda t}\sum_{i=1}^N \xi'_iM_i(\xi_i,a)\ge
\frac{1}{\lambda t}\sum_{i=1}^N \xi'_iM_i(\xi'_i,a)\ge \widetilde{M}(t',a).
$$
which yields the nondecreasing monotonicity. In particular the limit in (\ref{meanM}) does exist. Denote it by $\widetilde{\mu}$. Since $M_k(\xi_k,a)\ge \mu_k(a)\ge \mu(a)$, we have  $\widetilde{M}(t,a)\ge \mu(a)$. In particular, $\widetilde{\mu}\ge \mu(a)$. Conversely, given $t>0$ let $\xi\in S(t)$ be the corresponding minimum point of (\ref{not20}). Let the number $k=k(a)$, $1\le k\le N$, be chosen such that $\mu(a)=M_k(0,a)$. Define $\eta_i=0$ for $i\ne k$ and $\eta_k=t$. Then
$$
\widetilde{M}(t,a)=\frac{1}{t}\sum_{i=1}^N \xi_iM_i(\xi_i,a)\le \frac{1}{t}\sum_{i=1}^N \eta_iM_i(t,a)=M_k(t,a).
$$
Passing to the limit as $t\to +0$ in the latter inequality yields $\widetilde{\mu}\le \mu(a)$, thus implying (\ref{meanM}).

Finally, assume again that $\xi\in S(t)$ is the minimum point of (\ref{not20}) for $t>0$. Then using \ref{Hain2} and the H\"older inequality we obtain
\begin{align*}
t(\widetilde{M}(t,a)-\mu(a))&=\sum_{i=1}^N(M_i(\xi_i,a)-\mu(a)) \xi_i\ge \sum_{i=1}^N(M_i(\xi_i,a)-M_i(0,a)) \xi_i\\
&\ge p(a)\sum_{k=1}^N\xi_i^{1+\gamma}\ge
\frac{p(a)}{N^\gamma}(\sum_{i=1}^N\xi_i)^{1+\gamma} =\frac{p(a)}{N^\gamma}t^{1+\gamma},
\end{align*}
which yields (\ref{mtilde}).
\end{proof}

\begin{proposition}\label{pro:upper}
In the notation of Proposition~\ref{prop:estsigma}, if $\sigma(\Rol)>1$ then there exists a unique $\theta_+>0$ such that
\begin{align}\label{r+}
\int_0^{\infty} \frac{m(a)\,e^{-\int_0^a\mu(s)\,ds}}{(1+\theta_+^\gamma P(a))^{1/\gamma}}= 1,
\end{align}
where
$$
 P(a)=\frac{\gamma}{N^\gamma} \int_{0}^{a}p(t)e^{-\int_0^t \mu(s)ds}dt.
$$
Furthermore,
\begin{align}\label{est:up}
\sum_{k=1}^N\theta_k \le \theta_+.
\end{align}
\end{proposition}

\begin{proof}
Since $\sigma(\Rol)>1$, the maximal solution $\theta\gg0$ and  $\theta=\Kop\theta$.  Let $\varphi_k(a,\theta)$ denote the corresponding solution of (\ref{varphieq}) satisfying (\ref{char111}). Let $\psi(a)=\sum_{k=1}^N\varphi_k(a;\theta)$. Then summing up equations (\ref{varphieq})  and using (\ref{cond:D}) and (\ref{not20}) we obtain
\begin{align*}
\frac{d }{da}\psi(a)&\le -\sum_{k=1}^NM_k(\varphi_k(a;\theta),a) \varphi_k(a;\theta) \le -\widetilde{M}(\psi(a),a)\psi(a),
\end{align*}
The obtained inequality implies that $\psi(a)$ is a (positive) decreasing function of $a\ge0$, in particular, $0<\psi(a)\le \psi(0)=\|\theta\|_{\infty}$. We have from (\ref{mtilde})
$$
\frac{d}{da}\psi(a)+\mu(a)\psi(a)\le -(\widetilde{M}(\psi(a),a)-\mu(a))\psi(a)
\le -\frac{p(a)}{N^\gamma}\psi(a)^{1+\gamma}.
$$
Rewriting the obtained inequality for $z(a)=\psi(a)\exp(\int_0^a\mu(s)\,ds)$as
$$
\frac{dz(a)}{da}\le -\frac{p(a)}{N^\gamma}z(a)^{1+\gamma}e^{-\gamma\int_0^a\mu(s)\,ds},
$$
yields after integrating
$$
\frac{1}{z(a)^\gamma}-\frac{1}{z(0)^\gamma}\ge \frac{\gamma}{N^\gamma} \int_{0}^{a}p(t)e^{-\gamma\int_0^t \mu(s)ds}dt=P(a)
$$
This yields by virtue of $z(0)=\psi(0)=\|\theta\|_{\infty}$
\begin{equation}\label{Q1}
\psi(a)\le \frac{\|\theta\|_{\infty}\,e^{-\int_0^a\mu(s)\,ds}} {(1+P(a)\|\theta\|_{\infty}^\gamma )^{1/\gamma}}
\end{equation}
Next, since $\theta=\Kop\theta$, it readily follows that
$$
\|\theta\|_{\infty}\le \int_0^{\infty}m(a)\psi(a)\,da \le \|\theta\|_{\infty} \int_0^{\infty} \frac{m(a)\,e^{-\int_0^a\mu(s)\,ds}}{(1+ P(a)\|\theta\|_{\infty}^\gamma)^{1/\gamma}}.
$$
This yields by virtue of $\|\theta\|_{\infty}>0$ that
$$
\int_0^{\infty} \frac{m(a) \,e^{-\int_0^a\mu(s)\,ds}}{(1+P(a)\|\theta\|_{\infty}^\gamma )^{1/\gamma}}\ge 1.
$$
Since the integral
$$
I(t)=\int_0^{\infty}\frac{m(a) \,e^{-\int_0^a\mu(s)\,ds}}{(1+P(a)t^\gamma )^{1/\gamma}}
$$
is a decreasing function of $t$ and $\lim_{t\to \infty}I(t)=0$, there exists (a unique) $\theta_+\ge \|\theta\|_{\infty}$ solving the equation (\ref{r+}), thereby proving (\ref{est:up}).
\end{proof}

\begin{remark}
Let us comment on (\ref{est:up}) from the biological point of view. Notice by Theorem~\ref{th:est} that $\sum_{k=1}^N\theta_k $ is the asymptotical value of the total number  of newborns on all patches. By the dichotomy, $\sigma(\Rol)\le1$ implies $\theta=0$, thus the total asymptotical number of newborns is zero. On the other hand, in the  nontrivial case $\sigma(\Rol)>1$, hence  by (\ref{est:s}) $\int_0^{\infty}m(a)e^{-\int_0^a\mu(v)dv}\,da>1$, which easily  implies that (\ref{r+}) has a positive solution.
\end{remark}

The next proposition provides a lower estimate for the maximal solution.

\begin{proposition}\label{pro:lower}
Let there exist a function $q(a)>0$ such that
\begin{equation}\label{hain20}
\begin{split}
M_k(v,a)-M_k(0,a)&\le q(a) v^{\gamma}, \quad \forall (v,a)\in \R{2}_+.
\end{split}
\end{equation}
If for some $k$
\begin{equation}\label{lowere}
\int_0^{\infty}m_k(a)e^{-\int_0^a(\mu_k(v)+|D_{kk}(v)|)dv}\,da>1
\end{equation}
then
\begin{align}\label{est:low}
\theta^-_k \le \theta_k,
\end{align}
where $\theta^-_k$ is the unique solution to equation
\begin{align}\label{r-}
\int_0^{\infty}\frac{m_k(a)e^{-\gamma\int_0^a(\mu(s))+|D_{kk}(a)|)\,ds}} {(1+(\theta^-_k)^\gamma Q(a))^{1/\gamma}}\,da=1,
\end{align}
and
$$
 Q(a)=\frac{\gamma}{N^\gamma} \int_{0}^{a}q(t)e^{-\int_0^t (\mu(s)+|D_{kk}(s)|)ds}dt.
$$

\end{proposition}

\begin{proof}
First notice that (\ref{lowere}) implies by (\ref{est:s}) that $\sigma(\Rol)>1$, thus $\theta\gg0$. Since $D_{kj}(a)\ge0$ for $j\ne k$, the $k$-th equation in (\ref{varphieq}) yields
$$
\frac{d }{da}\varphi_k(a;\theta) \ge -(M_k(\varphi_k(a,\theta),a)-D_{kk}(a))\varphi_k(a;\theta),
$$
hence using (\ref{hain20}) we obtain by virtue of $M_k(0,a)=\mu_k(a)$ that
$$
\frac{d }{da}\varphi_k(a;\theta)+(\mu_k(a)-D_{kk}(a))\varphi_k(a)\ge
-q(a)\varphi_k(a)^{1+\gamma}.
$$
Arguing similar to the proof of Proposition~\ref{pro:upper} we get from $\varphi_k(0;\theta)=\theta_k$ that
\begin{equation}\label{Q2}
\varphi_k(a;\theta)\ge \theta_k\frac{\,e^{-\int_0^a(\mu(s)+|D_{kk}(a)|)\,ds}}{(1+\theta_k^\gamma Q(a))^{1/\gamma}},
\end{equation}
therefore
$$
\theta_k=(\Kop(\theta))_k \ge  \theta_k \int_0^{\infty}\frac{m_k(a)e^{-\gamma\int_0^a(\mu(s))+|D_{kk}(a)|)\,ds}} {(1+\theta_k^\gamma Q(a))^{1/\gamma}}\,da.
$$
Since $\theta\gg 0$, one has $\theta_k>0$, hence
$$
\int_0^{\infty}\frac{m_k(a)e^{-\gamma\int_0^a(\mu(s))+|D_{kk}(a)|)\,ds}} {(1+\theta_k^\gamma Q(a))^{1/\gamma}}\,da\le 1.
$$
Again, let
$$
I(t)=\int_0^{\infty}\frac{m_k(a)e^{-\gamma\int_0^a(\mu(s))+|D_{kk}(a)|)\,ds}} {(1+\theta_k^\gamma Q(a))^{1/\gamma}}\,da.
$$
Then $I(t)$ is decreasing, $I(\theta_k)\le1$ and by (\ref{lowere}) $I(0)>1$, thus there exists (a unique) solution $\theta^{-}_k$ of (\ref{r-}) such that $\theta_k\ge \theta^{-}_k$.
\end{proof}

\section{Periodically varying environment}\label{sec:periodic}

Now we consider an important particular case of the main problem (\ref{genpr})--(\ref{genic}) when the environment is periodically changing. In this section and in the rest of the paper, it is assumed that the vital rates, regulating function and dispersion coefficients are time-dependent and periodic with a period $T>0$. The boundary-initial value problem (\ref{genpr})--(\ref{genic}) is now in a $T$-periodic domain $\B$, where $B(t+T)=B(t)$, $t\in \R{}$, under the periodicity assumption that
\begin{equation}\label{periodic}\begin{aligned}
\mathbf{m}(a,t+T) &= \mathbf{m}(a,t), \\
\mathbf{M}(v,a,t+T) &= \mathbf{M}(v,a,t),\\
\mathbf{D}(a,t+T) &= \mathbf{D}(a,t)
\end{aligned}\end{equation}
for any $1\le k,j\le N$. Throughout this section, we  assume that the conditions \ref{Hain1}--\ref{Hain5} are satisfied.

Notice that the existence and uniqueness of a solution $\mathbf{n}(a,t)$ to the periodic problem follows from the general result given by Proposition~\ref{prsol} and it is given explicitly by (\ref{n1}). Note also that $\mathbf{n}(a,t)$ need not to be periodic in $t$ but it is natural to expect that $\mathbf{n}(a,t)$ converges to a $T$-periodic function $\rho(t)$ for  $t$ sufficient large, where $\rho(t)$ solves the associated \textit{characteristic equation}
\begin{align}\label{charT}
\Kopp\rho(t)=\rho(t), \quad t\in\mathbb{R}.
\end{align}
Here the operator $\Kopp$ is defined by
\begin{align*}
\Kopp\rho(t):=\int_0^{\infty}\mathbf{m}(a,t)\Phi (a;\rho,t-a)\,da, \quad t\in\mathbb{R}, \quad 1\leq k\leq N
\end{align*}
and $\Phi(x;\rho,y)$ denotes the (unique) solution $h(x)$ of the initial value problem
\begin{equation}\label{newPDE}
\begin{aligned}
\renewcommand\arraystretch{1.5}
\left\{
\begin{array}{rll}
\frac{d}{dx}h(x) &=&-\mathbf{M}(h(x),x,x+y)h(x) +\sum_{j=1}^N\mathbf{D}(x,x+y)h(x),
 \\
  h(0) &=& \rho(y),
\end{array}
\right.
\end{aligned}
\end{equation}
where the initial condition
$$
\rho\in C_T(\R{}_+,\R{N}_+):=\{\rho\in C(\R{}_+,\R{N}_+):\rho(t+T)=\rho(t)\}.
$$

We shall assume that the nonnegative cone $C_T(\R{}_+,\R{N}_+)$ is equipped with the supremum norm $\|\rho(t)\|_{C([0,T])}$. It follows from the uniqueness results of section~\ref{subMain} that the function $\Phi(x;\rho,y)$ is $T$-periodic in $y$.

A function $\rho\in C_T(\R{}_+,\R{N}_+)$ is said to be an \textit{upper} (resp. \textit{lower}) solution to  (\ref{charT}) if $\rho\ge\Kopp \rho$ (resp. $\rho\le\Kopp \rho$). It follows from Lemma \ref{propPhi} and  condition~\ref{Hain4}, it follows that $\Kopp$ has a bounded range:
\begin{align}\label{Omegap}
\|\Kopp(\rho)\|_{C([0,T])}\le \omega_2.
\end{align}
In particular, any solution of the characteristic equation (\ref{charT}) is bounded by $\omega_2$.

Recall that a (nonlinear) operator is called \textit{absolutely continuous} if it is continuous and maps bounded sets into relatively compact sets.

\begin{lemma}\label{uniform}
$\Kopp:C_T(\R{}_+,\R{N}_+)\to C_T(\R{}_+,\R{N}_+)$ is an absolutely continuous operator.
\end{lemma}

\begin{proof}
By the Arzela-Ascoli theorem it suffices to show that the family of functions
$$
\{\Kopp\rho: \rho\in C_T(\R{}_+,\R{N}_+) \,\,\text{and } \|\rho(t)\|_{C([0,T])}\le R\}
$$
is uniformly bounded and equicontinuous for any $R>0$. The first property is by (\ref{Omegap}). In order to prove that the family  is equicontinuous, we estimate $|\Kopp\rho(t_1)-\Kopp\rho(t_2)|$ for $|t_1-t_2|<\delta_1$ and for any $\rho\in C_T(\R{}_+,\R{N}_+)$ such that $\|\rho(t)\|_{C([0,T])}\le R$. To this end, we assume that $\tau:=t_2-t_1>0$ is such that
$$
\tau<\delta_1<\half12\min \{a_m,b_1-A_m\},
$$
where $a_m,$ $A_m$ and $b_1$ are the structure constants in \ref{Hain1} and \ref{Hain4}. Rewriting
\begin{align*}
\Kopp\rho(t_2) &=\int_{a_m}^{A_m}\textbf{m}(a,t_2)\Phi(a,t_2-a;\rho)da \\
&=\int_{a_m-\tau}^{A_m-\tau}\textbf{m}(a+\tau,t_1+\tau)\Phi(a+\tau,t_1-a;\rho)da
\end{align*}
and using the property that $ \textbf{m}(a,t_i)=0$ for any $a$ outside $[a_m,A_m]$ for $i=1,2$ we obtain component-wise
\begin{align*}
|(\Kopp\rho(t_1))_k-(\Kopp\rho_k(t_2))_k| & \le  \int_{a_m/2}^{A_m}|m_k(a+\tau,t_1+\tau )-m_k(a,t_1)|\Phi_k(a+\tau,t_1-a;\rho)\,da\\
&+\int_{a_m/2}^{A_m} m_k(a,t_1)|\Phi_k(a+\tau,t_1-a;\rho)-\Phi_k(a,t_1-a;\rho)|\,da\\
&=:I_1+I_2
\end{align*}
We have by (\ref{estphi}) that for any $\tau>0$ and $t_1\in\R{}$
$$
\int_{a_m/2}^{A_m}\Phi_k(a+\tau,t_1-a;\rho)\,da\le \int_{a_m/2}^{A_m}\frac{\omega_1}{(a+\tau)^{-1/\gamma}}\,da\le \omega_1\int_{a_m/2}^{A_m}\frac{1}{a^{-1/\gamma}}\,da=:
C_1
$$
where $C_1$ depends only on the structural constants.
Next,  since $m_k(a,t)$ is a $T$-periodic in $t$, by \ref{Hain4} $m_k$ is uniformly continuous on the strip $[a_m,A_m]\times \R{}$. Since $\supp m_k\subset [a_m,A_m]\times \R{}$, there exists $\delta_2>0$ such that for any $1\le k\le N$,  $a\in [0,A_m]$ and  $|\tau|<\delta_2$ one has the inequality
$$
|m_k(a+\tau,t_1+\tau)-m_k(a,t_1)| <\frac{\epsilon}{2C_1}.
$$
This yields $I_1<\epsilon/2$.
In order to estimate $I_2$, we notice that $\Phi_k(x):=\Phi_k(x,t_1-a;\rho)$  is the solution of the initial problem (\ref{newPDE}).  Notice that by (\ref{Phibound})
$$
\max_{1\le k\le N}\|\Phi_k\|_{C([0,b])}\le \sqrt{N}e^{N\|\mathbf{D}\|b}\|\rho\|_\infty\le C_2:=R\sqrt{N}e^{N\|\mathbf{D}\|b}.
$$
Let
\begin{align*}
C_3&:=\max \{M_k(v,a,t): 0\le v\le C_1, \, 0\le \half12(b_1+A_m), \, 0\le t\le T\}\\
&:=\max \{M_k(v,a,t): 0\le v\le C_1, \, 0\le \half12(b_1+A_m), \, 0\le t<\infty\},
\end{align*}
where the latter equality is by the periodicity. Therefore, applying the mean value theorem to (\ref{newPDE}) we obtain for any $0\le x_1<x_2<A_m+\delta_1$ and for some $\xi\in (x_1,x_2)$ that
$$
\frac{|\Phi_k(x_1)-\Phi_k(x_2)|}{x_2-x_1} \le
(|M_k(\Phi_k(\xi),\xi,t_1-a)|+ N\|\mathbf{D}\|)C_2\le (C_3+ N\|\mathbf{D}\|)C_2=:C_4,
$$
where $C_4$ depends only on the structure conditions and  $R$. This readily implies
$$
I_2\le C_4A_m\|\mathbf{m}\|_\infty\delta_1.
$$
Choosing $\delta_1$ small enough,  yields the desired conclusion.
\end{proof}

%

\begin{proposition}\label{max_p}
For any $\rho^+(t)$ such that  $\rho^+(t)\ge \omega_2\cdot \mathbf{1}_{N}$, where $\omega_2$ is defined by (\ref{omegadef}), the limit
\begin{align*}
\theta(t):=\lim_{i\to \infty }\Kopp ^i(\rho^+(t))
\end{align*}
exists and is a solution to the characteristic equation \eqref{charT}. Furthermore, the limit $\theta(t)$  does not depend on a particular choice of $\rho^+(t)$ and it is the maximal  solution to equation $(\ref{charT})$ in the sense that if $\rho(t)$ is any solution to the characteristic equation \eqref{charT} then $\rho(t)\le \theta(t)$. Furthermore, if $\rho^-(t)$ is a lower solution  then $\theta(t)\ge \rho^-(t)$.
\end{proposition}

\begin{proof}
Since  $\Kopp \rho^+(t)\le \omega_2 \cdot\textbf{1}_{N}\le \rho^+(t)$ and by the monotonicity of $\Kopp$ we get:
$$
\Kopp^{j+1}\rho(t)\equiv \Kopp^j\Kopp \rho^+(t)\le \Kopp^j \rho^+(t),
$$
which implies that $\{\rho^{(j)}(t)\}$ is a non-increasing sequence. The sequence is bounded from below because $\Kopp^j \rho^+\ge0$, therefore there exists a pointwise  $\lim_{j\to \infty} \Kopp^j \rho^+(t)=:\theta(t) $. The sequence $\{\rho^{(j)}(t)\}$ is uniformly bounded by the constant $\omega_2$. Applying Lemma~\ref{uniform} to family $\{\rho^{(j)}_k(t)\}$ implies that the convergence is in fact uniform on each compact subset of $\mathbb{R}$. Thus $\theta$ is a nonnegative continuous $T$-periodic  solution of $(\ref{charT})$.
 The rest of the proof is analogous to the proof of  Proposition \ref{it}.
\end{proof}


In the remaining part of this section we additionally assume that additionally condition \ref{Hain6} holds. In that case, due to the periodicity, the infimum in \ref{Hain6} can be replaced by the minimum.
Then arguing similarly to Lemma~\ref{lem:Ymono}, one can verify that for any $\rho(t)\in C_T(\R{}_+,\R{N}_+)$ and $0<\lambda_1<\lambda_2$,
$$
\frac{1}{\lambda_1}\Kopp (\lambda_1 \rho)\gg \frac{1}{\lambda_2}\Kopp (\lambda_2\rho),
$$
hence the corresponding net reproductive operator is well-defined defined by
\begin{align*}
\Rlo \rho=\lim_{\lambda\to +0}\frac{1}{\lambda}\Kopp (\lambda\rho)=\int_0^{\infty}\mathbf{m}(a,t)\mathbf{Y}(a; \rho,t-a)\,da,
\end{align*}
where $\mathbf{Y}(x,y;\rho)$ is the solution of  the linear system
\begin{equation*}\begin{aligned}
\frac{d \mathbf{Y}(x,y;\rho)}{dx}&=(\mathbf{D}(x,x+y)-\mathbf{M}(0,x,x+y)) \mathbf{Y}(x,x+y;\rho),
\\
\mathbf{Y}(0,y;\rho)&=\rho(y).
\end{aligned}\end{equation*}

Let $\sigma(\Rlo)$ denote the largest eigenvalue of $\Rlo$ and let $\theta=\theta(t)\in C_T(\R{}_+,\R{N}_+)$ be the maximal solution of equation (\ref{charT}).  Then the following results are established similarly to  Theorem \ref{th:main},  Theorem \ref{th:est} and Theorem \ref{th:estN} respectively.

\begin{theorem}
If $\sigma(\Rlo)\leq 1$, then the characteristic equation \eqref{charT} has no nontrivial solutions (in particular, $\theta\equiv 0$). If $\sigma(\Rlo)>1$, then $\theta\gg0$ is the only nontrivial solution of equation \eqref{charT}.
\end{theorem}

\begin{theorem}
\label{th:estp}
If $\mathcal{F}\mathbf{f}(t) \not\equiv 0$ and $\chi(t)$ is a solution to \eqref{n}  then $\lim_{t\rightarrow\infty}\chi(t)=\theta(t)$.
\end{theorem}

\begin{theorem}
Let $\mathbf{P}(t)=\int_0^{T}\mathbf{n}(a,t)\,da$ be the total multipatch population. If $\sigma(\Rlo)\le1$, then $\mathbf{P}(t)\rightarrow 0$ as $t\rightarrow\infty$. If $\sigma(\Rlo)>1$, then
\begin{align*}
\lim_{t\rightarrow\infty}\mathbf{P}(t) = \int_0^{\infty}\Phi(a,t-a;\theta)\,da,
\end{align*}
where $\theta$ is the maximal solution to the characteristic equation (\ref{charT}).
\end{theorem}

\section{Irregularly varying environment}\label{sec:irreg}

In order to study asymptotic behavior of the solution to the model (\ref{genpr})--(\ref{genic}) in the case when temporal variation is irregular, we assume  that the vital rates, regulating function and dispersion coefficients are bounded from below and above by equiperiodic functions for large $t$. These periodic functions define two auxiliary periodic problems, whose solutions provide upper and lower bounds to a solution of the original problem. This leads us to two-side estimates of a solution to the original problem for large $t$.

More precisely, throughout this section we shall suppose that there exists $T_1\ge0$ and $T$-periodic functions $m^{\pm}_k$, $M^{\pm}_k$ and $D^{\pm}_{kj}$ such that for any $a\ge0$ and $t\ge T_1$
\begin{equation}\label{pb}
\begin{aligned}
\mathbf{m}^-(a,t) &\le \mathbf{m}(a,t) \le \mathbf{m}^+(a,t), \\
\mathbf{M}^+(a,t) &\le \mathbf{M}(a,t) \le \mathbf{M}^-(a,t),\\
\mathbf{D}^-(a,t) &\le \mathbf{D}(a,t) \le \mathbf{D}^+(a,t).
\end{aligned}
\end{equation}

As in Section~\ref{sec:periodic}, one can consider the corresponding characteristic equations
\begin{align*}\label{charpKpm}
\Kopp^{\nu}\rho^{\nu}(t)=\rho^{\nu}(t), \quad t\in\R{},
\end{align*}
where $\nu$ denote $-$ or $+$, and the operators $\Kopp^{\nu}$ are defined component-wise by
\begin{align}
\Kopp^{\nu}\rho^\nu(t):=\int_0^{\infty}\mathbf{m}^{\nu}(a,t) \Phi^{\nu}(a,t-a;\rho^\nu)\,da, \quad t\in\R{}_+, \quad 1\leq k\leq N,
\end{align}
and $\Phi^{\nu}(x,y;\rho)$ is the unique solution of the system
\begin{equation*}\label{PDEpm}
\begin{split}
\frac{d\Phi^{\nu}(x,y;\rho)}{dx}&=-\mathbf{M}^{\nu}(\Phi^{\nu}(x,y;\rho),x,x+y) \Phi^{\nu}(x,y;\rho) +\mathbf{D}^{\nu}(x,x+y)\Phi^{\nu}(x,y;\rho),
 \\
\Phi^{\nu}(0,y;\rho) &= \rho(y),
\end{split}
\end{equation*}
with $\rho\in C_T(\R{}_+,\R{N}_+)$. Then by Proposition \ref{prsol}
\begin{align*}
\rho^{\nu}(t)=\mathcal{K}^{\nu}\rho^{\nu}(t)+\mathcal{F}^{\nu}\mathbf{f}(t),
\end{align*}
where
\begin{equation}
\begin{split}\label{k+}
\mathcal{K}^{\nu}\rho(t) &=\int_0^t\mathbf{m}^{\nu}(a,t)\Phi^{\nu}(a;\rho,t-a)\,da,\\
\mathcal{F}^{\nu}\mathbf{f}(t) &=\int_t^{\infty}\mathbf{m}^{\nu}(a,t)\Psi^{\nu}(a;\mathbf{f},a-t)\,da.
\end{split}
\end{equation}
Also let us denote by $\Rol^{\nu}$ and $\sigma(\Rol^{\nu})$ the corresponding net reproductive operators and net reproductive rates.   The main result of this section states that a solution of the population problem in an irregularly changing environment can be estimated by the corresponding solutions of the associated periodically varying population problems.

\begin{theorem}\label{main_gen}
Let $\chi(t)$ be a solution to equation $(\ref{n})$. Then the following dichotomy holds:
\begin{enumerate}
\item[(i)]
If $\sigma(\Rol^+)\le 1$, then $\lim_{t\to \infty}\chi(t)=0$.
\item[(ii)]
If $\sigma(\Rol^-)>1$ and $\mathcal{F}^-\mathbf{f}(t)\not\equiv 0$, then for any $\epsilon>0$ there exists $T_2>0$ such that
\begin{align}\label{est:pernb}
\rho^-(t)-\varepsilon\le \chi(t) \le \rho^+(t)+\varepsilon \qquad \forall t>T_2,
\end{align}
where $\rho^{\pm}(t)$ are solutions to $(\ref{charpKpm})$.
\end{enumerate}
\end{theorem}

\begin{proof}
Without loss of generality  $T_1\ge B(0)$. Let $R>2\omega_2$  and let us define $\{\chi^{(j)}(t)\}_{j\ge0}$ and $\{\chi^{(j)}_+(t)\}_{j\ge0}$ iteratively for $t>T_1$ by
\begin{equation*}
\begin{split}
&\chi^{(j+1)}(t)=\mathcal{K}\chi^{(j)}(t) + \mathcal{F}\mathbf{f}(t), \quad \chi^{(0)}(t)=R\\
&\chi^{(j+1)}_+(t)=\mathcal{K}^+ \chi^{(j)}_+(t), \qquad\qquad \chi^{(0)}_+(t)=R,
\end{split}
\end{equation*}
where the operators $\mathcal{K}$ and $\mathcal{F}$  are defined by (\ref{k}) and (\ref{f}) respectively.

Arguing as in the proof of Proposition~\ref{gen_sol}, we obtain the existence   $\lim_{j\to \infty}\chi^{(j)}(t)=\chi(t)$, where $\chi(t)$ is a solution to  (\ref{n}). Also by Proposition \ref{max_p}, $\lim_{j\to \infty}\chi^{(j)}_+(t)=\chi^+(t)$, where $\chi^+(t)$ is the maximal solution to (\ref{charpKpm}). We will prove by induction that for any $ j\ge 0$ there holds
\begin{align}\label{est:1}
\chi^{(j)}(t)\le \chi^{(j)}_+(t), \quad\forall t>T_1+jA_m.
\end{align}

For $j=0$ the claim follows from $\chi^{(0)}(t)=\chi^{(0)}_+(t)=R$ for $t>T_1$. Next, by our choice of $T_1$, $\mathcal{F}\mathbf{f}(t)=\mathcal{F}^+\mathbf{f}(t)=0$ for $t>T_1$.  Since $\Phi(0; \chi^{(0)},y)=\Phi^{\pm}(0; \chi^{(0)}_+,y)=R$ for any $y\ge0$ and the structure parameters are estimated by (\ref{pb}), one easily deduces from the definition of $\Phi^{\nu}(x,y;\rho)$ that  $\Phi(a; \chi^{(0)},t-a)\le \Phi^{+}(a; \chi^{(0)}_+,t-a)$ for $a\ge0$ and $t-a\ge T_1$. Since
$$
\chi^{(1)}(t) = \int_0^t\mathbf{m}(a,t)\Phi(a;\chi^{(0)},t-a)\,da=
\int_{a_m}^{A_m}\mathbf{m}(a,t)\Phi(a;\chi^{(0)},t-a)\,da
$$
and $t-a>T_1$ for all $a\in [a_m,A_m]$ and $t>T_1+A_m$ we obtain
$$
\chi^{(1)}(t) = \int_0^t\mathbf{m}(a,t)\Phi(a;\chi^{(0)},t-a)\,da \le  \int_0^t\mathbf{m}^+(a,t)\Phi^+(a;\chi^{(0)}_+,t-a)\,da = \chi^{(1)}_+(t).
$$
This proves the induction assumption for $j=1$. Now suppose that the induction claim holds for some $j\ge 1$. Arguing similarly, we obtain for any  $t>T_1+(j+1)A_m$ that
$$\chi^{(j+1)}(t) = \mathcal{K}\chi^{(j)}(t) \le \mathcal{K}\chi^{(j)}_+(t) \le \mathcal{K}^+\chi^{(j)}_+(t)= \chi^{(j+1)}_+(t),$$
which proves (\ref{est:1}). Therefore, passing to the limit we obtain
\begin{align}\label{est:r}
\chi(t)=\lim_{j\rightarrow\infty}\chi^{(j)}(t) \le \lim_{j\rightarrow\infty}\chi_+^{(j)}(t) =\chi^+(t).
\end{align}
If $\sigma(\Rol^+)\le 1$, then by Theorem \ref{th:estp}  $\lim_{t\to\infty}\chi^+(t)=0$, hence (\ref{est:r}) implies (i).

To proceed with (ii) notice that (\ref{est:r}) already yields the upper estimate in (\ref{est:pernb}). It remains to show that there exists a lower solution $\chi^-(t)$ to $\chi(t)=\mathscr{L}_\mathbf{f}^-\chi(t)$. We use auxiliary sequence $\{\chi^{(j)}_-(t)\}$ given by
$$\chi^{(j+1)}_-(t)=\mathscr{L}_\mathbf{f}^-\chi^{(j)}(t), \quad \chi^{(0)}_-(t)=0,$$
to define function
\begin{equation}\label{lowso}
\renewcommand\arraystretch{1.5}
\chi^-(t) =\left\{
\begin{array}{l}
\chi^{(j)}_-(t), \quad 0\le t\le T_1\\
\lambda\rho^-(t), \quad t>T_1,
\end{array}
\right.
\end{equation}
where  $\rho^-$ is a solution to the characteristic equation $\Kopp^-\rho^-(t)=\rho^-(t)$ and $\lambda>0$ is sufficiently small.

Notice first that the sequence $\{\chi^{(j)}_-(t)\}$  is nondecreasing in $j$ and that each $\chi^{(j)}_-(t)$ satisfies $\chi^{(j)}_-(t)\le\mathscr{L}_\mathbf{f}^-\chi^{(j)}(t) $, i.e., it is a lower solution to equation $\chi(t)=\mathscr{L}_\mathbf{f}^-\chi(t)$. Hence, $\chi^-(t)$ defined by (\ref{lowso}) is a lower solution in the interval $t\in[0,T_1]$ for sufficiently large $j$. Now suppose that $t\in[T_1,T_1+A_m]$. By \ref{Hain1} we have  $\mathcal{F}^-\mathbf{f}(t)=0$ and $\mathscr{L}_\mathbf{f}^-\chi^-(t)=\mathcal{K}^-\chi^-(t)$. Thus,  (\ref{charpKpm}) and (\ref{lowso}) imply that
\begin{align}
&\mathcal{K}^-\chi^-(t) - \chi^-(t)= \int_0^{A_m}\mathbf{m}^-(a,t)\Phi^-(a;\chi_-,t-a)\,da - \lambda\rho^-(t) \nonumber\\
& \quad= \int_0^{t-T_1}\mathbf{m}^-(a,t)(\Phi^-(a;\lambda\rho^-,t-a) -\lambda\Phi^-(a;\rho^-,t-a))\,da \label{lowa1}\\
&\quad\quad + \int_{t-T_1}^{A_m}\mathbf{m}(a,t)(\Phi(a;\chi^{(j)}_-,t-a) -\lambda\Phi^-(a;\rho^-,t-a))\,da. \label{lowa2}
\end{align}
Arguing similarly to the proof of Theorem \ref{th:est}, yields that integrals (\ref{lowa1}) and (\ref{lowa2})  are nonnegative. This proves that $\chi^-(t)\le\mathscr{L}^-\chi^-(t)$ for $t\in[T_1,T_1+A_m]$.

For $t>T_1+A_m$, we have that $\mathcal{F}^-\mathbf{f}(t)=0$, and $\mathscr{L}_\mathbf{f}^-\chi^-(t)=\mathcal{K}^-\chi^-(t)$, hence
$$(\mathcal{K}^-\chi^- - \chi^-)_k(t) = \int_0^{A_m}m_k^-(a,t)(\Phi_k^-(a;\lambda\rho^-,t-a)-\lambda\Phi_k^-(a;\rho^-,t-a))\,da \ge 0.
$$
This proves that function $\chi^-(t)$ defined by (\ref{lowso}) is a lower solution of equation $\chi(t)=\mathscr{L}^-\chi(t)$. Therefore,
\begin{align}\label{est:2}
\chi^-(t) \le \chi(t), \quad t\ge 0.
\end{align}

If $\sigma(\Rol^-)>1$, then $\sigma(\Rol^+)>1$ and characteristic equations (\ref{charpKpm}) have nontrivial solutions $\rho^{\pm}(t)$.
Then by virtue of Theorem \ref{th:estp},  $\lim_{t\rightarrow\infty}\chi^-(t) = \rho^-(t)$ and  $\lim_{t\rightarrow\infty}\chi^+(t) = \rho^+(t)$. Passing to the limit in (\ref{est:1}) and (\ref{est:2}) yields (\ref{est:pernb}).
\end{proof}

\section{Applications }

In this section we consider two simple applications of our approach showing how dispersion promotes survival of a population on sink patches. In the usual situation, a habitat is a mixture of sources and sinks. Our first example shows that permanency on all patches is possible if the patches are connected and if emigration from sources is sufficiently small and does not cause extinction of a local subpopulation. Some researchers  indicate that survival of migrating species is possible even if all occupied patches are sinks, see  \cite{JanY}. Taking migratory birds as an example, we demonstrate that this is possible under certain conditions.

\subsection{A single source and multiple sinks}\label{sec:sink-source}

In order to demonstrate the influence of dispersion on persistence of population, we compare a system with $N$ {isolated} patches with the corresponding system with dispersion. Recall that in the isolated case, $\mathbf{D}(a,t)\equiv 0$ implying by (\ref{sigmaind}) that  the net reproductive rate of the $k$th patch is given  by
$$
\sigma_k=\int_0^{\infty}m_k(a)\Pi_k(v)\,da,
$$
where $\Pi_k(v)=e^{-\int_0^a\mu_k(v)\,dv}$ is the survival probability.

In this case the spectrum of the net reproductive operator is
$$
\spectrum(\mathscr{R}_{0})=\{\sigma_1,\ldots,\sigma_N\}.
$$
We assume that $\sigma_1>1$ and $\sigma_k\le 1$, for $k\ge 2$. In the biological terms, this is equivalent to saying that the first patch is a source and all other patches are sinks. Without migration, the population will persist on the first patch and become extinct on all other patches. For details about the age-structured logistic model that we used to describe isolated patches, we refer readers to \cite{we2}. Under the made assumptions,
\begin{align*}
\lim_{t\to\infty}\rho_1(t)&=\rho^*_1, \\
\lim_{t\to\infty}\rho_k(t)&=0, \quad 2\le k\le N,
\end{align*}
where $\rho^*_1>0$ is uniquely determined by
$$
\int_0^{\infty}\frac{m_1(a)\Pi_1(v)}{1+\rho^*_1(1-\Pi_1(v))}\,da=1.
$$

Now let us allow a small migration between patches and assume  that there also holds $\sigma_1>1$ and $\sigma_k\le 1$, for $k\ge 2$. Let us suppose that the dispersion coefficients
$$
\mathbf{D}(a)=\varepsilon \mathbf{B}(a),
$$
where $\varepsilon>0$ is a small number and the parameters $B_{kj}(a)$ satisfy \ref{Hain3} in Section~\ref{sec:form}. Then the standard linearization argument shows that the  solution to the corresponding time-independent model
\begin{equation}\label{exsys}
\begin{aligned}
\frac{d\varphi(a;\rho)}{da} &= -\mathbf{M}(a)\varphi(a;\rho) +\epsilon \mathbf{B}(a)\varphi(a;\rho), \quad \varphi(0;\rho) &=\rho,
\end{aligned}
\end{equation}
is given by
$$\varphi_k(a;\rho)=\Pi_k(v)\left( \rho_k + \varepsilon \int_0^a \sum_{j=1}^N\rho_jB_{kj}(s)\frac{\Pi_j(s)}{\Pi_k(s)}\,ds\right) + O(\varepsilon^2).$$
Therefore, the net reproductive operator takes the form
\begin{align*}
(\mathscr{R}_{0}\rho)_k &= \sigma_k\rho_k + \varepsilon\int_0^{\infty}m_k(a)\Pi_k(v)\int_0^a \sum_{j=1}^N\rho_jB_{kj}(s)\frac{\Pi_j(s)}{\Pi_k(s)}\,ds\,da + O(\varepsilon^2).
\end{align*}
Then latter relation yields
$$
\mathscr{R}_{0}=\diag(\sigma_1,\ldots,\sigma_N)+\epsilon\mathscr{B}+O(\varepsilon^2)
$$

Now, recall that if $A$ is a symmetric matrix and $x$ is an eigenvector with a simple eigenvalue $\lambda$ then the corresponding perturbed eigenvalue of $A+\epsilon B$ ($B$ may not be symmetric) is given by
$$
\lambda+\epsilon \mu+O(\epsilon^2), \qquad \mu=x^tBx/|x|^2.
$$
For $\varepsilon=0$, the largest eigenvalue is $\sigma_1$ with the eigenvector  $e_1=(1, 0,..., 0)$. The perturbed eigenvalue, which will be the net reproductive rate for the net reproductive operator~$\mathscr{R}_{0}$, is
$$\sigma(\mathscr{R}_{0})=\sigma_1+\varepsilon\int_0^{\infty}m_1(a)\Pi_1(v)B_{11}(a)\,da + O(\varepsilon^2),$$
and this is greater than one for small $\varepsilon>0$ provided that $B_{11}(a)\le 0$ and strictly negative in at least one point of the support of $m_1$. Thus shows that survival on all patches is possible if emigration from the source is sufficiently small.

\subsection{Multiple sinks, without a source}\label{sink-only}

Now consider the extreme situation when a population inhabits two patches and the net reproductive rate on \textit{each} patch is less or equal to one. We will demonstrate that, even in this case, there is a chance of survival if the structure parameters are suitably chosen.

A realistic example for this kind of situation is a population of migratory birds.  Their habitats consists of two patches: breeding range (characterized by the high birth rate in summer and high death rate in winter) and non-breeding range (low birth and death rates). Thus, the breeding range is a sink because of the winter conditions, and the non-breeding range is a sink because of too few births.

Let the death rates $\mu_1>\mu_2>0$ be constant on the supports $\supp m_1=[c_1,d_1]$ and $\supp m_2=[c_2,d_2]$, respectively, where $c_i,d_i$ will be chosen later. In addition, suppose that
$$
\sigma_k=\int_{c_k}^{d_k}m_k(a)e^{-\mu_ka}\,da=1, \quad k=1,2,
$$

This implies extinction of population on both patches if there is no dispersal.
If the dispersion matrix $D$ satisfies
$$D=\varepsilon B, \quad B= \left(
\begin{array}{cc}
-1 & 1 \\
1& -1
\end{array} \right),$$
then the solution to the system (\ref{exsys}) for $N=2$ is given by
$$\varphi_k(a;\rho)=e^{-\mu_k a}(\rho_k+\varepsilon h_k(a,\rho)+O(\varepsilon^2)), \quad k=1,2,$$
where
\begin{equation*}
\renewcommand\arraystretch{1.5}
\left\{
\begin{array}{ccc}
\frac{dh_1(a,\rho)}{da} = -\rho_1+e^{(\mu_1-\mu_2)a}\rho_2, \quad h_1(0)=0, \\
\frac{dh_2(a,\rho)}{da} =-\rho_2+e^{(\mu_2-\mu_1)a}\rho_1, \quad h_2(0)=0.
\end{array}
\right.
\end{equation*}
A solution to this system is given by
\begin{equation*}
\renewcommand\arraystretch{1.5}
\left\{
\begin{array}{ccc}
h_1(a,\rho)=-\rho_1a+\frac{1}{\mu_1-\mu_2}(e^{(\mu_1-\mu_2)a}-1)\rho_2, \\
h_2(a,\rho)=-\rho_2a+\frac{1}{\mu_2-\mu_1}(e^{(\mu_2-\mu_1)a}-1)\rho_1.
\end{array}
\right.
\end{equation*}
Then, the net reproductive operator satisfies
\begin{align*}
(\mathscr{R}_{0}\rho)_k &= \rho_k   + \varepsilon\int_{c_k}^{d_k}m_k(a)e^{-\mu_ka}h_k(a,\rho)\,da + O(\varepsilon^2), \quad k=1,2.
\end{align*}
In the matrix form this becomes
\begin{align}
\Rol\rho= \rho+\varepsilon\mathscr{P}\rho
+O(\varepsilon^2\rho),
\end{align}
where
\begin{align*}
\renewcommand\arraystretch{1.5}
\mathscr{P}=\left(
  \begin{array}{c c}
    -1 & \int_{c_1}^{d_1}\frac{m_1(a)e^{-\mu_1a}(e^{(\mu_1-\mu_2)a}-1)}{\mu_1-\mu_2}\,da \\
   \int_{c_2}^{d_2} \frac{m_2(a)e^{-\mu_2a}(e^{(\mu_2-\mu_1)a}-1)}{\mu_2-\mu_1}\,da  & -1
\end{array}\right).
\end{align*}
Thus, to show that $\sigma(\Rol)>1$, it is sufficient to show that $\mathscr{P}\rho>0$ for some choice of parameters and certain vector $\rho$.

Using
$$\psi(z)=z^{-2}(e^z-1-z)=\frac{1}{2}+\frac{z}{3!}+\frac{z^2}{4!}+...$$
it follows that the functions $h_1$ and $h_2$ can be written as:
\begin{equation*}
\renewcommand\arraystretch{1.5}
\left\{
\begin{array}{ccc}
h_1(a,\rho)=(\rho_2-\rho_1)a+a^2(\mu_1-\mu_2)\psi((\mu_1-\mu_2)a)\rho_2, \\
h_2(a,\rho)=(\rho_1-\rho_2)a+a^2(\mu_2-\mu_1)\psi((\mu_2-\mu_1)a)\rho_1.
\end{array}
\right.
\end{equation*}
Sine the function $z\psi(z)$ monotonically increases from $0$ to $\infty$, there exists a unique $c^*$ such that $c^*(\mu_1-\mu_2)\psi((\mu_1-\mu_2)c^*)=1$.
Suppose that $d_2<c^*<c_1$. Let us choose parameters  $\rho_1>\rho_2>0$ such that $h_1(a,\rho)>0$ for $a>c_1$ and $h_2(a,\rho)>0$ for $a<d_2$, that is
\begin{equation*}
\renewcommand\arraystretch{1.5}
\left\{
\begin{array}{ccc}
\rho_1-\rho_2 < a(\mu_1-\mu_2)\psi((\mu_1-\mu_2)a)\rho_2, \quad\mbox{for $a>c_1$}, \\
\rho_1-\rho_2 > a(\mu_1-\mu_2)\psi((\mu_2-\mu_1)a)\rho_1, \quad\mbox{for $a<d_2$},
\end{array}
\right.
\end{equation*}
or equivalently,
\begin{equation*}
\renewcommand\arraystretch{1.5}
\left\{
\begin{array}{ccc}
\frac{\rho_1}{\rho_2}-1< a(\mu_1-\mu_2)\psi((\mu_1-\mu_2)a), \quad\mbox{for $a>c_1$},\\
1-\frac{\rho_2}{\rho_1}> a(\mu_1-\mu_2)\psi((\mu_2-\mu_1)a), \quad\mbox{for $a<d_2$}.
\end{array}
\right.
\end{equation*}
We put $\rho_1=1$ and choose  $\rho_2<\frac{1}{2}$ and  $c_1$ and $d_2$ as solutions to equations:
$$\frac{1}{\rho_2}- 1 = c_1(\mu_1-\mu_2)\psi((\mu_1-\mu_2)c_1)$$
and
$$1-\rho_2 = d_2(\mu_1-\mu_2)\psi((\mu_2-\mu_1)d_2).$$
It follows that $\mathscr{P}\rho>0$ and hence $\Rol\rho>\rho$. The latter implies that that $\sigma(\Rol)>1$, thus $\Rol$ has an eigenvalue greater than one, which proves the permanency of population on both patches.









\frenchspacing
\bibliographystyle{cpam}

\def\cprime{$'$}

\end{document}